\documentclass[english,11pt]{amsart}
\usepackage{amssymb,amsmath,amsthm}
\usepackage{ mathrsfs }
\usepackage{srcltx}
\usepackage{babel,graphicx}
\usepackage{color}
\usepackage{enumitem}
\usepackage{mathtools}
\usepackage{xcolor}
\usepackage{amsmath,amsthm,amssymb,amsfonts,mathtools,csquotes}
\usepackage{tikz-cd}
\usepackage[backend=biber,style=alphabetic,maxalphanames=5,sorting=nyt, maxbibnames=99]{biblatex}
\usepackage{comment}
\date{November 28 2025}
\usepackage[colorlinks,citecolor=blue,urlcolor=,bookmarks=false,hypertexnames=true]{hyperref}

\newtheorem{theorem}{Theorem}[section]
\newtheorem{corollary}[theorem]{Corollary}
\newtheorem{proposition}[theorem]{Proposition}
\newtheorem{lemma}[theorem]{Lemma}
\newtheorem{assumption}[theorem]{Assumption}

\theoremstyle{definition}
\newtheorem{remark}[theorem]{Remark}
\newtheorem{example}[theorem]{Example}
\newtheorem{definition}[theorem]{Definition}

\newtheorem{transverse lyapunov exponent strictly maximally negative when r(G)=2 cases}{Case}
\newtheorem{transverse lyapunov exponent strictly maximally negative for least singular elements cases}{Case}
\newtheorem{case negativity of transversal lyapunov exponent on positive weyl chambers}{Case}
\newtheorem{case no transverse manifolds in bad directions of generic points}{Case}
\newtheorem{case empirical measures sequence}{Case}
\newtheorem{step main}{Step}

\newenvironment{topological measure rigidity}{\emph{Proof of Proposition \ref{topological measure rigidity}:}}{\hfill$\square$}
\newenvironment{transverse lyapunov exponent strictly maximally negative when r(G)=2}{\emph{Proof of Proposition \ref{transverse lyapunov exponent strictly maximally negative} when $r(G)=2$:}}{\hfill$\square$}
\newenvironment{transverse lyapunov exponent strictly maximally negative for least singular elements}{\emph{Proof of Proposition \ref{transverse lyapunov exponent strictly maximally negative for least singular elements}:}}{\hfill$\square$}
\newenvironment{proof of P-equivariance case r(G)>2}{\emph{Proof of Proposition \ref{P-equivariance of transverse Pesin manifolds}:}}{\hfill$\square$}
\newenvironment{proof of entropic measure rigidity}{\emph{Proof of Proposition \ref{entropic measure rigidity}:}}{\hfill$\square$}
\newenvironment{proof of coherence proposition}{\emph{Proof of Proposition \ref{coherence proposition}:}}{\hfill$\square$}
\newenvironment{proof of coherence proposition when rG=2} {\emph{Proof of Proposition \ref{coherence proposition} when $r(G)=2$:}}{\hfill$\square$}
\newenvironment{proof of Pertubations by generic elements of P have LU factorization}{\emph{Proof of Proposition \ref{Pertubations by generic elements of $P$ have $LU$ factorization}:}}{\hfill$\square$}
\newenvironment{proof of N intersects all orbits}{\emph{Proof of Proposition \ref{N intersects all orbits}:}}{\hfill$\square$}
\newenvironment{proof of main}{\emph{Proof of Theorem \ref{main}:}}{\hfill$\square$}
\newenvironment{proof of entropic measure rigidity negative transverse exponents}{\emph{Proof of Proposition \ref{entropic measure rigidity} when all elements in $\Lambda_\mu^T(s)$ are negative:}}{\hfill$\square$}
\newenvironment{sketch of proof}{\emph{Sketch of Proof}:}{\hfill$\square$}

\newcommand{\N}{\mathbb{N}}
\newcommand{\Z}{\mathbb{Z}}

\newcommand{\R}{\mathbb R}

\newcommand{\s}{\vspace{0.2cm}}
\newcommand{\m}{\mathcal M}

\newcommand{\supp}{\operatorname{supp}}

\newcommand{\lo}{\lambda_{\mu^P}^T}
\newcommand{\haar}{\operatorname{Haar}}
\newcommand{\SL}{\operatorname{SL}}
\newcommand{\vol}{\operatorname{Vol}}

\newcommand{\locstable}{\mathcal W^s_{\operatorname{loc}}}
\newcommand{\lie}{\operatorname{Lie}}
\newcommand{\intW}{\operatorname{Int}(\mathcal W^P)}
\newcommand{\hmup}{h_{\mu^P}}

\def\acts{\curvearrowright}
\def\actsright{\curvearrowleft}

\makeatletter
\@namedef{subjclassname@2020}{%
  \textup{2020} Mathematics Subject Classification}
\makeatother

  \DeclareFontFamily{U}{wncy}{}
    \DeclareFontShape{U}{wncy}{m}{n}{<->wncyr10}{}
    \DeclareSymbolFont{mcy}{U}{wncy}{m}{n}
    \DeclareMathSymbol{\Sha}{\mathord}{mcy}{"58}

\title{Global Rigidity of Codimension One Actions}
\author{Camilo Arosemena Serrato}
\address{Department of Mathematics\\Rice University\\Houston, TX 77005}
\email{ja102@rice.edu}
\begin{document}
\begin{abstract}
    Consider a smooth, locally free, codimension-one action of a higher-rank, simple, split Lie group $G$ on a closed manifold $M$. Let $P$ be a minimal parabolic subgroup of $G$. If the action admits a $P$-invariant probability measure that is mixing, then the action is either equivariantly diffeomorphic to the suspension of a codimension one, locally free action on a closed manifold of a parabolic subgroup of $G$; or, it is finitely and equivariantly covered by $G\acts G/\Gamma\times S^1$, where $G\acts G/\Gamma$ is the coset action, and $G\acts S^1$ is the trivial action. We prove this by doing a jointly integration argument of stable and  center unstable Pesin manifolds. This is a smooth version of results by Nevo and Zimmer. 
\end{abstract}
\maketitle
\section{Introduction}

An essential part of the Zimmer program, laid out by Robert Zimmer in his 1986 ICM address, see \cite{zimmericm}, is to classify actions of semisimple, higher rank Lie groups on closed manifolds. Let $G$ be such a Lie group, and $P$ a minimal parabolic subgroup of $G$. A significant result in this vein is the  Theorem of Nevo and Zimmer, see \cite{nzinventiones}, which proves that for any measurable action of $G$ on a compact metric space $X$ for which there exists a $P$-mixing measure $\lambda$ on $X$, either $\lambda$ is $G$-invariant, or there exists a parabolic subgroup $Q$ of $G$ such that the original action of $G$ on $X$ is the suspension action of a measurable, measure preserving action of $Q$ on a probability space(see Definition \ref{suspension space defintion} for the notion of suspension space). The authors weaken the hypothesis of this result in their subsequent work, see \cite[Theorem~3]{nzann}, by only assuming there exists a $P$-invariant probability measure for which an specific finite set of lines of $A$ are ergodic with respect to this measure, where $A$ is any maximal torus of $G$ contained in $P$.

This theorem by Nevo and Zimmer does not answer the question of whether a continuous (or smooth) action of $G$ on a metric space (or smooth manifold)  is the suspension action of a continuous(or smooth action) of a parabolic subgroup of $G$ on some metric space(or smooth manifold), when the action has no invariant probability measure. In this paper, we answer this question for an important class of such actions, namely, for locally free, codimension one actions of $G$. Let $G$ be a simple, higher rank and split Lie group; $P$ a minimal parabolic subgroup of $G$; $A$ a maximal torus of $G$ contained in $P$; and $\Pi_P$ the set of simple roots of $(G,A)$ with respect to $P$. The following theorem is the main result of this paper:

 \begin{theorem}\label{main}
     Suppose $G$ acts locally freely and in a $C^2$ manner on a closed connected manifold $M$, with codimension $1$ orbits. Suppose there exists a $P$-invariant probability measure $\mu^P$ on $M$ such that $\ker(\alpha)\acts (M,\mu^P)$ is ergodic, for all $\alpha\in\Pi_P$. Then exactly one of the following holds:
    \begin{enumerate} 
        \item $M$ has an invariant probability measure for the action $G\acts M$. Furthermore, $M$ is equivariantly and finitely covered by the diagonal action $G\acts G/\Gamma\times S^1$, for some uniform lattice $\Gamma\leq G$, where $G\acts S^1$ is the trivial action;
        \item\label{suspension statement} There exists a proper parabolic subgroup $Q\leq G$ and an equivariant smooth map $\pi:M\rightarrow G/Q$. Furthermore, $M$ is equivariantly diffeomorphic to the suspension space $G\times_Q N$, where $N=\pi^{-1}(\{Q\})\subseteq M$ is an embedded submanifold invariant under the action of $Q$, and $\dim(N)=\dim(Q)+1$.
    \end{enumerate}
\end{theorem}

One crucial idea in the proof of this result is an integrability argument of foliations by $C^2$ manifolds, whose transversal regularity is only measurable, see Lemma \ref{injectivey immersed N}. Such integration arguments in dynamical systems are challenging tasks, and are the center of a lot of current research, see the work of Katz \cite{katz}, and the recent work of Brown, Eskin, Filip, and Rodr\'iguez-Hertz, see \cite{brown2025measurerigiditygeneralizedugibbs}.

Another essential tool for the proof of this result are topological dynamical arguments appearing in the recent work of Deroin and Hurtado, see \cite{dh}, that adapt constructions and results of smooth dynamical systems, to the context of dynamical systems of continuous functions.

Important motivation for this result has been the recent celebrated work, in the Zimmer program, on the resolution of several important cases of Zimmer's Conjecture, on classification of actions of lattices of higher rank Lie groups on manifolds of small dimension with respect to the Lie group, by Brown, Fisher and Hurtado, see \cite{bfh}; and the aforementioned recent proof of Deroin and Hurtado of the non-left-orderability of this class of lattices, see \cite{dh}. Locally free actions play a pivotal role in these results via the suspended action construction(for this notion, see Definition \ref{suspension space defintion}), which yields a locally free action of the Lie group encoding the original lattice action; a main feature of these results is using the algebraic structure of the Lie group, and the geometric structure of the induced action, to apply tools from dynamical systems and geometry to prove properties of the suspended action, which then correspond to properties of the original lattice action.

Locally free actions give rise to a foliation of the acted upon manifold, with leaves the orbits of this action. A notable class of foliations, for which a lot of work in foliation theory has been done, are those whose leaves have codimension one in the manifold, i.e., codimension one foliations. Thus, an important class of locally free actions of Lie groups is that for which their orbits are of codimension one in the manifold. The study of this kind of actions has had various important results, as we illustrate below.

An important early observation on the rigidity of locally free actions of codimension one is the fact that the only orientable closed surface admitting a one-dimensional smooth flow without singularities, is the torus, because of the Hopf-Poincar\'e index theorem. In the early seventies, the work of  Rosenberg, Roussarie, and  D.Weil, and then the work of Chatelet and Rosenberg, see \cite{rrw} and \cite{chateletrosenberg}, proves that the only closed oriented manifolds of dimension $n$ admitting a locally free action of $\R^{n-1}$ are fiber bundles over $\mathbb T^k$ with fiber $\mathbb T^{n-k}$. 

In his 1979 PhD thesis Ghys, see \cite{ghys85}, proves that all smooth actions of the group $\operatorname{AG}$, of orientation preserving affine transformations of the real line(which is the only non-abelian two-dimensional Lie group), on 3-manifolds, are smoothly conjugate to a homogeneous action, whenever there is  a continuous volume form invariant under the action. Later, Asaoka, see \cite{asaokagrenoble}, proved that all smooth actions of $\operatorname{AG}$ are smoothly orbit equivalent to a homogeneous action; and later that there are such actions that are not smoothly conjugate to a homogeneous action, see \cite{asaokaannals}.

For semisimple Lie groups $G$ with property (T), Stuck and Zimmer proved that any locally  free, $C^2$ action $G\acts M$, with codimension one orbits, having a $G$-invariant probability measure, is finitely and 
equivariantly covered by the diagonal action $G\acts G/\Gamma\times S^1$, where $G\acts G/\Gamma$ is the coset action and $G\acts S^1$ is the trivial action, see \cite{sz} and Theorem \ref{sz}. The assumption on the existence of such invariant probability measures is quite strong as $G$ is not amenable. For instance, showing such invariant measures exist, for the corresponding suspension space, is the main argument in the theorem by Ghys on finiteness of smooth actions of  lattices of higher rank semisimple Lie groups on the circle, see \cite{ghys}.

Our main result, Theorem \ref{main}, improves the results of Stuck-Zimmer, in the sense that we need a much weaker hypothesis, and we prove that all the actions of the same class are of algebraic origin.

%In latter work of the same authors, under the strong assumption of the existence of an ergodic, fully supported, stationary measure of a smooth action $G\acts M$, they show the existence of an equivariant map $M_0\rightarrow G/Q$, for some parabolic subgroup $Q$ of $G$, where $M_0$ is an open, dense, $G$-invariant  subset of $M$.

%The first part of Theorem \ref{main} follows from the Theorem of Stuck and Zimmer, see Theorem \ref{sz}. The main focus of this paper will thus be in obtaining (2) in Theorem \ref{main}, from assuming $G\acts M$ has no invariant probability measures, see Assumption \ref{main assumption}.    

%\begin{remark}
 %   We will prove that the only possible parabolic subgroups appearing in Theorem \ref{main} are either the minimal parabolic, or [FINISH THIS]
%\end{remark}!!!!!!!!!!!!!!!!!!!!!!!!!!!!!!!!!!!!!!!!!!

\subsection{Outline of the proof}

The first item of the conclusion of Theorem \ref{main} follows from the  Theorem by Stuck and Zimmer, see Theorem \ref{sz}. Hence, to prove Theorem \ref{main}, we need only prove the second item of it by assuming the indicated action $G\acts M$, satisfying the hypothesis of this theorem, has no invariant probability measure. 

Fix $P\leq G$ minimal parabolic, and $A$ a maximal torus of $G$ with $A\leq P$. The first step in proving Theorem \ref{main} is to prove, under the assumpion that $G\acts M$ has no invariant probability measures, that the convex set of all $P$-invariant, Borel, probability measures on $M$ has only finitely many extremal points, see Proposition \ref{uniquePergodicity}. This is done by applying the main result of the paper \cite{dk} of Deroin and Kleptsyn on harmonic probability measures of foliations with codimension one leaves; then showing that stationary measures with respect to the locally free action $G\acts M^{\dim(G)+1}$ are harmonic measures for the foliation of $M$ whose leaves are the orbits of this action, and then applying the classical work of Furstenberg, see \cite{furstenberg}, yielding a convex bijection between stationary and $P$-invariant probability measures for $G$-actions. 

Fix $\mu^P$ a $P$-invariant measure satisfying the ergodic hypothesis of our main result, Theorem \ref{main}, so that in particular it is $A$-ergodic. 

By the higher rank Osedelets' Theorem, see \cite[Theorem~2.4]{brhwzd}, there exists a Lyapunov exponent with respect to $\mu^P$, whose corresponding Oseledets' distribution is, $\mu^P$-almost surely, transversal to the $G$-orbits. We denote this Lyapunov exponent, which is an element of $A^\ast$, by $\lo$. In section \ref{Lyapunov Exponent Transverse to Orbits and Resonance}, we show that $\lo$ is less than all roots of $(G,A)$, in the interior of the Weyl chamber of $A$ associated to $P$, see Proposition \ref{transverse lyapunov exponent strictly maximally negative}. This yields, for $\mu^P$ almost every $x$, a one-dimensional $C^2$, injectively immersed, Pesin submanifold, which we denote by $\mathcal W^T(x)$, transverse to the $G$-orbit $G\cdot x$ at $x$, for $\mu^P$-a.e. $x$. 

Denote by $Q$ the stabilizer subgroup of $\mu^P$, that is,
\begin{equation*}
    Q=\{g\in G:g_\ast\mu^P=\mu^P\},
\end{equation*}
where $g_\ast\mu^P$ is the pushforward measure of $\mu^P$ under the diffeomorphism of $M$ given by the action of $g$, i.e., the map $M\ni x\mapsto gx$. 

In section \ref{Proof of theorem when DeltaminusDeltaQgeq3} we prove, when $G$ is not locally isomorphic to $\SL(3,\R)$, that  for $\mu^P$-a.e. $x$ the submanifold $\mathcal W^T(x)$ is $\haar_Q$ almost everywhere equivariant with respect to the action $Q\acts M$, see Proposition \ref{P-equivariance of transverse Pesin manifolds}. This property is used to show that 
\begin{equation*}
    Q\cdot\mathcal W^T(x)=\{qy:q\in Q, y\in\mathcal W^T(x)\}
\end{equation*}
is an injectively immersed smooth submanifold of $M$ of dimension $\dim(Q)+1$, which we denote by $N$. $N$ does not depend on any $x$ in a subset of $M$ of full $\mu^P$ measure. By construction, the restriction of the original action to $Q$ yields a locally free action of this subgroup on $N$.

When $G$ is locally isomorphic to $\SL(3,\R)$, the same construction as above can be made, see Section \ref{Construction of submanifold when r(G)=2}, yielding an injectively immersed submanifold as well, by showing $\mu^P$ disintegrates, along the family $\{\mathcal W^T(x)\}_{x\in\Lambda}$, for some  $\Lambda\subseteq M$ with $\mu^P(\Lambda)=1$, into measures that are equivalent to the Riemannian volume on $\mathcal W^T(x)$, with strictly positive corresponding $C^1$ Radon-Nikodym derivatives, see Proposition \ref{absolute continuity along transverse Pesin}; then, the construction of the corresponding $N$ is shown to be an injectively immersed submanifold by using the absolute continuity property of the holonomy maps between $G$-orbits, along the family $\{\mathcal W^T(x)\}_{x\in\Lambda}$, see Section \ref{Construction of submanifold when r(G)=2}. 

We construct the $G$-equivariant map $M\rightarrow G/Q$ in Theorem \ref{main}, by showing the following properties about the submanifold $N$:

\begin{enumerate}
    \item\label{important property of N I} $N$ intersects all the orbits of the action $G\acts M$;
    \item\label{important property of N II} for all $x\in M$, $G\cdot x\cap N=gQ\cdot x$, for a unique $gQ\in G/Q$.
\end{enumerate}
These two properties are shown in Section \ref{N is a closed embedded submanifold}, see Lemma \ref{N intersects all orbits}, by using a fact we call ``Topological Measure Rigidity", proved in Section \ref{section topological measure rigidity}, see Proposition \ref{topological measure rigidity}, motivated by the recent work of Deroin and Hurtado, see \cite[Theorem~8.2]{dh}. The ``Topological Measure Rigidity" property for $\mu^P$, says that in a set of full measure of $\mu^P$, $\Lambda$, we have that if for some $x\in\Lambda$, $ux\in\supp(\mu^P)$ for some $u\in G^{-\beta}\setminus\{e\}$, for any $\beta$ root of $G$ with respect to the torus $A$, positive with respect to $P$, then $\mu^P$ is $G^{-\beta}$-invariant. For the proof of this fact, the ergodic assumption in Theorem \ref{main} on $\mu^P$ is essential.

The ``Topological Measure Rigidity" property implies property (\ref{important property of N II}) above, see Proposition \ref{coherence proposition}; it implies property (\ref{important property of N I}) by showing $P\acts M$ is actually uniquely ergodic, using that there are finitely many extremal points in the set of $P$-invariant probability measures, see the Proof of Proposition  \ref{N intersects all orbits}, at the end of Section \ref{N intersects all G-orbits section}. 

From properties (\ref{important property of N I}) and (\ref{important property of N II}), we finish the proof of Theorem \ref{main} by showing $M$ is equivariantly diffeomorphic to the suspension space $G\times_QN$, see the proof of Theorem \ref{main} in Subsection \ref{Construction of smooth equivariant map}, and see Defintion \ref{suspension space defintion} for the notion of suspension space. 

\subsection{Acknowledgments}
This work is part of the author's PhD thesis.

The author thanks Sebasti\'an Hurtado for all his feedback, suggestions, and enthusiasm for this project, as well as for visits to Yale University; his help was essential to this project. The author thanks his advisor, David Fisher, for his unwavering guidance and support throughout this project, as well as for the discussions, suggestions, and feedback that informed the ideas presented in this work. The author also thanks Aaron Brown, Ralf Spatzier, and Kurt Vinhage for their many helpful suggestions, discussions, and feedback, as well as for visits to their institutions.  The author would also like to thank Rose Elliot Smith and Miguel Pineda for helpful suggestions and discussions.

\tableofcontents

\section{Preliminaries}

\subsubsection{The theorem of Pugh and Shub on $\R^k$ ergodic measures}

We have the following theorem which proves that for a probability measure ergodic under the action of $\R^k$, for any $k\geq 1$, almost all lines in $\R^k$ going through the origin act ergodically as well. 

\begin{theorem}\cite[Theorem~1]{pughshub}\label{pughshub}
    If $\mathbb R^k$ acts ergodically on $(M,\mu)$, $\mu(M)<\infty$, and $L^2(M,\mu)$ is separable, then all elements in $\mathbb R^k$ off a countable family of hyperplanes act $\mu$-ergodically on $M$.
\end{theorem}

\subsection{Lyapunov exponents and Pesin manifolds}

\subsubsection{Osedelec's multiplicative ergodic theorem and Lyapunov exponents}

For any $C^1$ diffemorphism $f:M\rightarrow M$, and any $f$-ergodic probability measure $\mu$ on $M$ we have Osedelec's theorem, see \cite{osedelec}, and also  \cite[Theorem~7.1]{brown2019entropylyapunovexponentsrigidity}.

\begin{theorem}[Osedelec \cite{osedelec}]\label{Osedelec's theorem}
There are:
\begin{enumerate}
    \item a measurable set $\Lambda$ with $\mu(\Lambda)=1$;
    \item real numbers $\lambda^1>\ldots>\lambda^p$;
    \item a $\mu$-measurable, $Df$-invariant splitting $T_xM=\bigoplus_{i=1}^r E^i(x)$, for all $x\in\Lambda$,
\end{enumerate}
    such that for all $x\in\Lambda$ and all $v\in E^i(x)\setminus\{0\}$ we have
    \begin{equation*}
        \lim_{n\rightarrow\pm\infty}\frac1n\log||D_xf^n(v)||=\lambda^i.
    \end{equation*}
\end{theorem}

The real numbers $\lambda^i$ are called \emph{Lyapunov exponents}, and the spaces $E^i(x)$ are called Osedelets' distributions. We call the $\mu$-almost surely value $m^i:=\dim E^i(x)$ the multiplicity of $\lambda^i$, for $\mu$-a.e. $x$.

\subsubsection{Unstable Pesin manifolds}\label{unstable pesin manifolds}

Recall that given a compact manifold $M$, a $C^2$ diffeomorphism $f:M\rightarrow M$, and a $f$-ergodic probability measure $\mu$, with corresponding Lyapunov exponents $\lambda^1>\ldots>\lambda^d$; there are associated global $j$th unstable manifolds, which are  $C^2$ injectively  immersed 
submanifolds given by 
\begin{equation}\label{Pesin manifolds}
    \mathcal W^j(x)=\left \{y\in M:\limsup_{n\rightarrow\infty}\frac1n\log d(f^{-n}(x),f^{-n}(y))\leq -\lambda^j\right\}
\end{equation}
whenever $\lambda^i>0$, for $\mu$-a.e. $x$. These manifolds satisfy $T_x\mathcal W^j(x)=\bigoplus_{\lambda^j\leq\lambda^i}E^i(x)$ for $\mu$-a.e. $x$.

The set 
\begin{equation*}
    \mathcal W^u(x)=\left \{y\in M:\limsup_{n\rightarrow\infty}\frac1n\log d(f^{-n}(x),f^{-n}(y))< 0\right\}
\end{equation*}
is equal to the $r_0$th unstable manifold, where $1\leq r_0\leq p$ is such that $\lambda_{r_0}<0$, but $\lambda_{r_0+1}\geq 0$, if it exists, so that $\mathcal W^u(x)$ is a $C^2$ injectively immersed submandifold, for $\mu$-a.e. $x$, called \emph{the unstable Pesin manifold} at $x$. It satisfies $T_x\mathcal W^u(x)=\bigoplus_{\lambda^i<0}E^i(x)=:E^u(x)$, for $\mu$-a.e. $x$.

We denote by $\mathscr W^j$ and $\mathscr W^u$ the corresponding partitions by Pesin manifolds.  

There are also corresponding stable and central distributions $E^s(x):=\bigoplus_{\lambda^i<0}E^i(x)$, $E^0(x):=\bigoplus_{\lambda^i=0}E^i(x)$, only the former  has Pesin manifolds $\mathcal W^s(x)$ tangent to $E^s(x)$

\subsection{Entropy and The Ledrappier-Young theorem}

In this subsection we follow \cite[Section~8]{brown2019entropylyapunovexponentsrigidity}.

Let $f$ be a diffeomorphism of a closed manifold $M$. Let $\mu$ be an ergodic, $f$-invariant, probability measure on $M$. The metric entropy of $f$ with respect to $\mu$, denoted by $h_\mu(f)$, is equal to the entropy conditional to any measurable increasing  partition $\xi$ of $M$, subordinate to $\mathscr W^u$, this value is defined in \cite[Section~8.1.2]{brown2019entropylyapunovexponentsrigidity}, such measurable partitions always exists, and this value is independent of the measurable partition satisfying this properties, see \cite[Section~8.3]{brown2019entropylyapunovexponentsrigidity}.

For the $f$-ergodic probability measure $\mu$, consider  corresponding Lyapunov exponents $\lambda^1>\cdots>\lambda^p$, along with distributions $E^i(-)$ corresponding to $\lambda^i$, for $1\leq i\leq p$, as in Osedelec's Theorem, see Theorem \ref{Osedelec's theorem}. Denote by $mi$ the $\mu$-almost surely value $\dim E^i(x)$.

We have the following theorems relating the metric entropy of $f$ with respect to $\mu$, and the $\mu$-almost surely constant values given by Osedelec's Theorem.  

\begin{theorem}[Ruelle-Margulis inequality]\label{Ruelle-Margulis inequality}
    \begin{equation*}
        h_\mu(f)\leq \sum_{\lambda^i>0}m^i\lambda^i
    \end{equation*}
\end{theorem}

The $f$-ergodic probability measure $\mu$ is called an SRB measure, if for any measurable partition $\xi$ of $M$ subordinate to $\mathscr W^u$, we have that for $\mu$-a.e. $x$, the conditional measure $\mu^\xi_x$ is absolutely continuous with respect to the Riemannian volume on $\mathcal W^u(x)$.

The Ledrappier-Young Theorem, see \cite[Theorem~A]{ledrappieryoung1} says that the $f$-ergodic probability measures $\mu$ having an equality in Theorem \ref{Ruelle-Margulis inequality}, are exactly the SRB measures. That SRB measures have such an equality was first proved by Ledrappier and Strelcyn, see \cite{ledrappierstrelcyn}. Furthermore, in case of such equality, the corresponding Radon-Nikodym densities are $C^1$ and strictly positive, see the remark after \cite[Theorem~A]{ledrappieryoung1}.

\begin{theorem}[Ledrappier-Young theorem, \cite{ledrappieryoung1}]\label{ledrappier young}
    We have that
    \begin{equation*}
        h_\mu(f)= \sum_{\lambda^i>0}m^i\lambda^i,
    \end{equation*}
    if and only if for any measurable partition $\xi$ subordinate to $\mathscr W^u$ we have that $\mu_x^\xi$ is equivalent to the Riemannian volume on  $\mathcal W^u(x)$, for $\mu$-a.e. $x$. Furthermore,  the corresponding densities are positive  $C^1$ functions.
\end{theorem}

\subsubsection{Entropy conditional on measurable foliations}
Let $f:M\rightarrow M$ be a $C^2$ diffeomorphism, and $\mu$ an ergodic $f$-invariant probability measure on $M$. 

Let us consider measurable $f$-invariant foliations $\mathcal F$ of $M$ by $C^2$ leaves, which are \emph{tame measurable foliations}, that is, $\mathcal F$ is a partition of $M$ by $C^2$ manifolds, such that restricting the foliation to sets of large measure, $\mathcal F$ has the structure of a continuos family of $C^2$ disks. Examples of such partitions are the partition into Pesin manifolds, $\mathscr W^j$ and $\mathscr W^u$ constructed above, see \cite[Chapter~7]{bpintro}.

We say $\mathcal F$ is \emph{expanding} if $\mathcal F(x)\subseteq\mathcal W^u(x)$ for $\mu$-a.e. $x$. 

For any $\xi$ measurable partition subordinate to an expanding tame measurable foliation $\mathcal F$, we define the entropy of $f$ with respect to $\mu$ conditional on $\mathcal F$ to be
\begin{equation*}
    h_\mu(f\mid\mathcal F):=h_\mu(f\mid\xi),
\end{equation*}
and such measurable partitions always exist, and the value above does not depend on the particular measurable partition subordinate to $\mathcal F$.

If $\mathcal F$ is a tame measurable foliation, we set 
\begin{equation*}
    h_\mu(f,\mathcal F)=h_\mu(f,\mathcal F^u),
\end{equation*}
where $\mathcal F^u$ is the expanding, tame measurable partition given by 
\begin{equation*}
    \mathcal F^u:=\mathcal F\vee\mathscr W^u:=\{A\cap B: A\in\mathcal F, B\in \mathscr W^u\}.
\end{equation*}

\subsubsection{Conditional entropy along measurable partitions and the Ledrappier-Young Theorem}

We follow \cite[Section~4, Section~7]{brhwzd}.

We have the following version of the Ledrappier-Young inequality for conditional entropies, see \cite[Theorem~7.2]{brhwzd} and \cite[Corollary~6.1.4]{ledrappieryoung1}. Define the multiplicity of $\lambda^i$ relative to $\mathcal F$ to be the $\mu$-almost surely constant value of
\begin{equation*}
    m_i(\mathcal F):=\dim(E^i(x)\cap T_x[\mathcal F(x)])
\end{equation*}

\begin{theorem}\label{ledrappier young conditional entropy}
    We have
\begin{equation*}
    h_\mu(f|\mathcal F)\leq\sum_{1\leq i\leq r}\lambda_i m_i(\mathcal F).
\end{equation*}
Moreover, equality holds if and only if for any measurable partition $\xi$ subordinate to $\mathcal F^u$ we have that $\mu_x^\xi$ is equivalent to the Riemannian volume on  $\mathcal F^u(x)$, for $\mu$-a.e. $x$. Furthermore,  the corresponding densities are positive  $C^1$ functions.
\end{theorem}

\subsubsection{Geometric characterization of the defect in the entropy formula}

We follow \cite[Chapter~7]{brhwzd}. Let $f:M\rightarrow M$ be a $C^2$ diffeomorphism, and $\mu$ an ergodic, $f$-invariant probability measure on $M$. Let $\lambda^1>\ldots>\lambda^p$ be the 

Let $\eta$  be any measurable partition of $(M,\mu)$. For $1\leq i\leq r$, let $\xi^i$ be a measurable partition subordinate to the partition associated to the $i$th unstable Pesin manifolds, i.e., $\mathscr W^i$. We define the $i$th upper and lower pointwise dimensions of $\mu$ relative to $\eta$ at $x$ to be the limits
\begin{equation*}
    \overline{\dim}^i (\mu, x|\eta) := \limsup_{\delta \to 0} \frac{\log \left( \mu_x^{\xi^i \vee \eta}(B(x, \delta)) \right)}{\log \delta}, \underline{\dim}^i (\mu, x|\eta) := \liminf_{\delta \to 0} \frac{\log \left( \mu_x^{\xi^i \vee \eta}(B(x, \delta)) \right)}{\log \delta},
\end{equation*}
where $B(x,\delta)$ is the ball of radius $\delta$ centered at $x$ of a Riemannian
metric on $M$, these values are independent of such a metric, and of the chosen measurable partitions. These values are constant $\mu$-almost surely, and we  denote, respectively, these values by $\overline{\dim}^i(\mu\mid\eta)$ and $\underline{\dim}^i(\mu\mid\eta)$

For the partition $\eta^+:=\bigvee_{i=0}^\infty f^i\eta$, the values $\overline{\dim}^i(\mu\mid\eta^+)$ and $\underline{\dim}^i(\mu\mid\eta^+)$ coincide, see \cite[Proposition~7.4]{brhwzd}. 

Set $\dim^0(\mu|\eta^+) = 0$. For $1 \le i \le r$ the \textit{$i$th transverse dimension of $\mu$ relative to $\eta^+$} is
$$
\gamma^i(\mu\mid\eta^+):= \dim^i(\mu\mid\eta^+) - \dim^{i-1}(\mu\mid\eta^+).
$$
The $i$th transverse dimension is not bigger than the multiplicity $m_i$ of the Osedeltec's distributions $E^i$, see \cite[Claim~7.5]{brhwzd}.

Using the pointwise transverse dimensions above we can indicate the exact failure for equality to hold in the Ruelle-Margulis inequality (\ref{Ruelle-Margulis inequality}), this was proved by Ledrappier and Young.

\begin{theorem}\cite{ledrappieryoung2}\label{ledrappieryoung2}
    Let $\eta$ be any measurable partition of $(M,\mu)$. Then
    \begin{equation*}h_\mu(f\mid\eta)=\sum_{\lambda_i>0}\lambda_i\gamma^i(\mu\mid\eta^+).
    \end{equation*}
\end{theorem}

\subsection{Smooth ergodic theory of higher rank abelian actions}\label{Smooth ergodic theory of higher rank abelian actions}

In this subsection we follow \cite{brhwzd}.

\subsubsection{Higher-rank Osedelec's Theorem }

Let $M$ be a closed manifold and $\alpha:\Z^d\rightarrow\operatorname{Diff}^1(M)$ be an action. Let $\mu$ be an invariant and $\alpha$-ergodic probability measure. 

\begin{theorem}[Higher Rank Osedelec's Theorem]\cite[Theorem~2.4]{brhwzd}\label{Higher rank osedelec}
    There exist
    \begin{enumerate}
        \item an $\alpha$-invariant meausurable set $\Lambda\subseteq M$ such that $\mu(\Lambda)=1$;
        \item linear functionals $\lambda^1,\ldots,\lambda^p:\R^d\rightarrow\R;$
        \item a $\mu$-measurable, $D\alpha$ invariant splitting $T_xM=\bigoplus_{i=1}^p E^i(x)$, defined for $\mu$-a.e. $x$;
        \item for $\mu$-a.e. $x$ we have that for every $v\in E^i(x)\setminus\{0\}$, 
        \begin{equation*}
            \lim_{|n| \to \infty} \frac{\log \| D_x \alpha(n)(v) \| - \lambda^i(n)}{|n|} = 0
        \end{equation*}
    \end{enumerate}
\end{theorem}
We call the functionals $\lambda^1,\ldots,\lambda^p$ the Lyapunov exponents of $\alpha:\Z^d\acts (M,\mu)$; $E^i$ the Osedelec's distribution or vector space associated to $\lambda^i$; along with the almost surely constant multiplicity of $\lambda^i$, given by $m_i:=\dim E^i(x)$, for $\mu$-a.e. $x$.

\subsubsection{Global unstable Pesin manifolds in higher rank actions}\label{Global unstable Pesin manifolds in higher rank actions}
Let $\alpha:\Z^d\times (M,\mu)\rightarrow(M,\mu)$ be an ergodic action on a finite measure space. Let $\mathcal L=\{\lambda_i:\R^d\rightarrow\R:1\leq i\leq p\}$ be the set of corresponding Lyapunov exponents of this action. Given any $\vec n\in\Z^d$ there is a permutation $\Delta(\vec n)$ of $\{1,\ldots,p\}$ and $u(\vec n)$ in this set so that
\begin{equation*}
    \lambda_{\Delta(\vec n)(1)}(\vec n)\geq \lambda_{\Delta(\vec n)(2)}(\vec n)\geq\cdots\geq \lambda_{\Delta(\vec n)(u(\vec n))}(\vec n)>0\geq\cdots\geq\lambda_{\Delta(\vec n)(p)}(\vec n).
\end{equation*}

\begin{definition}\label{definition ith unstable}
    For any $x\in M$ and any $\vec n\in\Z^d$, and $1\leq i\leq u(\vec n)$, define the \emph{$i$th unstable manifold through $x$ for $\alpha(\vec n)$} as the set
\[
W^{u,i}_{\vec{n}}(x) = \left\{ y \in M \,\middle|\, \limsup_{k \to -\infty} \frac{1}{k} \log d\big(\alpha(k\vec{n}, x), \alpha(k\vec{n}, y)\big) \leq -\lambda_{\Delta(\vec{n})(i)}(\vec{n}) \right\}.
\]

The \textit{unstable manifold through \( x \) for \( \alpha(\vec{n}) \)} is the set
\[
W^u_{\vec{n}}(x) := \left\{ y \in M \,\middle|\, \limsup_{k \to -\infty} \frac{1}{k} \log d\big(\alpha(k\vec{n}, x), \alpha(k\vec{n}, y)\big) < 0 \right\}.
\]

\end{definition}

It can be shown that the unstable manifold $W^u_{\vec n}(x)$ with respect to the diffeomorphism $\alpha(\vec n)$ coincides with $\mathcal W^{u,u(\vec n)}_{\vec n}(x)$ for $\mu$-a.e. $x$. The manifold $W^{u,i}_{\vec{n}}(x)$ is tangent to $\bigoplus_{1\leq j\leq i}E_{\lambda_{\Delta(\vec n)(j)}}(x)$.

We have the following uniqueness property for the $i$th unstable manifolds defined above.

\begin{lemma}\cite[Lemma~4.7]{brhwzd}\label{uniqueness of ith unstable higher rank}
    Let \( \vec{n_1}, \vec{n_2} \in \mathbb{Z}^d \) have the following property: for some \( 1 \leq i \leq \min\{u(\vec{n_1}), u(\vec{n_2})\} \)
\begin{enumerate}
  \item \( \{\Delta(\vec{n_1})(j) : 1 \leq j \leq i\} = \{\Delta(\vec{n_2})(j) : 1 \leq j \leq i\} \),
  \item \( \lambda_{\Delta(\vec{n_1})(i)}(\vec{n_1}) > \lambda_{\Delta(\vec{n_1})(i+1)}(\vec{n_1}) \), \textit{and}
  \item \( \lambda_{\Delta(\vec{n_2})(i)}(\vec{n_2}) > \lambda_{\Delta(\vec{n_2})(i+1)}(\vec{n_2}) \).
\end{enumerate}
Then \( \mathscr{W}^{u,i}_{\vec{n_1}}\overset{\circ}{=}\mathscr{W}^{u,i}_{\vec{n_2}} \).
\end{lemma}

For two tame measurable foliations $\mathcal F$ and $\mathcal G$, the relation $\mathcal F\overset{\circ}{=}\mathcal G$ means that $\mathcal F(x)=\mathcal G(x)$ for $\mu$-a.e. $x$.

\subsubsection{Coarse Lyapunov exponents and coarse Pesin manifolds}\label{Coarse Lyapunov exponents and Pesin manifolds for higher rank actions}

For the $\alpha$-invariant and ergodic probability measure $\mu$, consider its set of Lyapunov exponents $\mathcal L:
=\{\lambda^i\}_{i=0}^p$. Two such Lyapunov exponents, $\lambda^i$ and $\lambda^j$,  are \emph{coarsely equivalent}, if there exists some $c>0$ such that $\lambda^i=c\lambda^j$. The corresponding equivalence classes are called \emph{coarse equivalent classes}, and we denote the corresponding set of equivalence classes by $\hat {\mathcal L}$. 

Given $\chi\in\hat{\mathcal L}$, with $\chi\neq\{0\}$(here $0$ is the zero element of $A^\ast$), the \emph{coarse Lyapunov exponent} corresponding to $\chi$ is given by 
\begin{equation*}
    \mathscr W^\chi:=\bigvee_{\{\vec n\in\Z^d:\chi(\vec n)>0\}}\mathscr W_{\vec n}^u.
\end{equation*}
With corresponding leaves denoted by $\mathcal W^\chi(x)$. $\mathscr W^\chi$ is a tame measurable foliation.

\subsubsection{Product structure of entropy for higher rank actions}

We have the following product structure theorem for higher rank actions.

\begin{theorem}\cite[Corollary~13.2]{brhwzd}\label{entropy product structure}
    For any $\vec n\in\Z^d$ we have 
    \begin{equation*}
        h_\mu(\alpha(\vec n))=\sum_{\{\chi\in\tilde{\mathcal L}:\chi(\vec n)>0\}}h_\mu(\alpha(\vec n)|\mathscr W^\chi)
    \end{equation*}
\end{theorem}

\subsection{The Theorem of Stuck and Zimmer for codimension one actions}

The following result by Nevo and Zimmer yields the first item for our main result, Theorem \ref{main}.

\begin{theorem}\cite[Theorem~5.1]{sz}\label{sz}
    Let $G$ be a higher rank, connected, simple, Lie group, with finite center. Suppose $G$ acts locally freely and in a $C^2$ manner on a closed connected manifold $M$, with codimension $1$ orbits. Assume there is a probability measure on $M$ invariant under this action.
    
    Then there exist a cocompact lattice $\Gamma$ in $G$ and a finite cover $\tilde M$ of $M$ such that $\tilde M$ is $G$-equivariantly diffeomorphic to $G/\Gamma\times S^1$.
\end{theorem}

\begin{remark}\label{idea of stuck zimmer}
    The proof of this theorem uses the existence of the Borel invariant probability measure to show that there exists a holonomy invariant measure on $M$, with respect to the foliation by $G$ orbits; this is used to show there is a closed orbit of the action, and from this it is concluded using the Reeb-Thurston theorem that $M$, and the fact that lattices of $\SL(n,\R)$ with $n\geq 3$ have property (T), must have a finite cover of the indicated form. 
\end{remark}

%we should repeat the proof here, so to work only in the C^1 case

\subsection{Invariance from entropy considerations}\label{Invariance from entropy considerations}

Here we follow \cite[\S~6]{el}, \cite{brown2019entropylyapunovexponentsrigidity}, and \cite {brhw}.

Let $G$ be a Lie group, and suppose $G$ acts, to the left, smoothly and locally freely on a closed manifold $M$. Suppose $H\leq G$ is a closed subgroup, and suppose $s\in G$ normalizes $H$. Let $\mu$ be an ergodic, $s$-invariant probability measure on $M$(we mean ergodicity and invariance with respect to the diffeomorphism of $M$ given by $x\mapsto sx$). Suppose the orbit $H\cdot x$ is contained in the unstable manifold $\mathcal W^u(x;s)$ for $\mu$-a.e. $x$. A measurable partition $\eta$ of $M$ is \emph{subordinate to the $H$-orbits of $M$} if for $\mu$-almost every $x$ we have that $\eta(x)$ is contained in $H\cdot x$; and contains a neighborhood of $x$ in the orbit $H\cdot x$. 

Since the action is locally free, we can pushforward (any positive multiple of)the left Haar measure on $H$ to the orbit $H\cdot x$, via the map $h\mapsto h\cdot x$.

\begin{lemma}\cite[Theorem~9.4]{brown2019entropylyapunovexponentsrigidity}\label{invariance for subgroups with orbits in unstable}
$\mu$ is $H$-invariant if and only if for any measurable partition $\xi$ subordinate to the partition into $H$-orbits, and for $\mu$-a.e. $x$ the conditional measure $\mu^\xi_x$ coincides(up to normalization) with the restriction of the left-Haar measure on $H\cdot x$ to $\xi(x)$.
\end{lemma}

We define the entropy of $s$, seen as an element of $\operatorname{Diff}(M)$, conditioned on $H$-orbtis, as the entropy conditioned on any measurable partition $\xi$ subordinate to the $H$-orbits, denoted by $h_\mu(s\mid H)$, namely,
\begin{equation*}
    h_\mu(s\mid H):=h_\mu(s\mid \xi).
\end{equation*}

Let $\lambda^i$, $E^i(x)$, and $m^i$ be as in \ref{Osedelec's theorem} for the dynamics of $s$ and the measure $\mu$.  We define the \emph{multiplicity of $\lambda^i$ relative to $H$} to be (the almost surely constant value of)  $$m^{i,H} = \dim (E^i(x) \cap T_x (H\cdot x)).$$

We have the following version of the Ledrappier-Young Theorem, see \cite{ledrappieryoung1}, characterizing $H$-invariant measures. 

\begin{lemma}\cite[Theorem~9.5]{brown2019entropylyapunovexponentsrigidity}\label{invariance from maximal entropy}
    The following are equivalent:
		\begin{enumerate} 
			\item $h_\mu(f\mid H) = \sum_{\lambda^i>0} \lambda^i m^{i,H}$;
			\item $\mu$ is $H$-invariant.  
		\end{enumerate}
\end{lemma}

To any Borel probability measure $\mu$ on $M$, there is a family $\{\mu_x^H\}_{x\in M}$ of measures of $H$, called \emph{leafwise measures along the orbits of the action} $H\acts M$, see \cite[Section~6]{el}, having the following property: for any measurable partition $\eta$ of $(M,\mu)$ subordinate to the $H$-orbits, there exists a measurable function $c^\eta:X\rightarrow (0,\infty)$ such that if $\{\mu_x^\eta\}_{x\in M}$ is a family of conditional measures on $(M,\mu)$ associated with $\eta$, then 
\begin{equation*}
    \mu_x^\eta=c^\eta(x)(h\mapsto h\cdot x)_\ast(\mu_x^H)\upharpoonright_{\eta(x)}.
\end{equation*}
We have the following immediate consequence of the previous lemma:
\begin{proposition}\cite[Problem~6.28]{el}\label{haarconditionals}
    A probability measure $\mu$ on $M$ is $H$-invariant if and only if $\mu_x^H$ coincides with the (positive proportionally class of the) left Haar measure on $H$ for $\mu$-almost every $x$.
\end{proposition}

\subsection{Smooth ergodic theory of nonuniform partially hyperbolic diffeomorphisms}\label{Smooth ergodic theory of nonuniformly partially hyperbolic diffeomorphisms}

We follow \cite{bpintro} and \cite{bp}. Let $f$ be a $C^1$ diffeomorphism of a closed manifold $M$. 

\subsubsection{Nonuniformly partially hyperbolic sets}

\begin{definition}\label{definition of nonuniformly partially hyperbolic subset}

We call an $f$-invariant nonempty measurable subset $Z \subset M$ \textit{nonuniformly partially hyperbolic} if there are a number $\varepsilon_0 > 0$, measurable functions $\lambda_1, \lambda_2, \mu_1, \mu_2, C, K : Z \to (0, \infty)$, and $\varepsilon : Z \to [0, \varepsilon_0]$, and subspaces $E^s(x)$, $E^u(x)$, and $E^0(x)$, which depend measurably on $x \in Z$, such that the following \textit{partial hyperbolicity conditions} hold:

\begin{itemize}
    \item[(1)] the functions $\lambda_1$, $\lambda_2$, $\mu_1$, $\mu_2$, $\varepsilon$ are $f$-invariant and satisfy for $x \in Z$
    \[
    \lambda_1(x)e^{\varepsilon(x)} < \lambda_2(x)e^{-\varepsilon(x)} < 1 < \mu_2(x)e^{\varepsilon(x)} < \mu_1(x)e^{-\varepsilon(x)};
    \]
    
    \item[(2)] for $m\in \Z$
    \begin{equation*}
        C(f^m(x))\leq C(x)e^{\varepsilon(x)|m|}, K(f^m(x))\geq K(x)e^{-\varepsilon(x)|m|}
    \end{equation*}
    \item[(3)] $T_x M = E^s(x) \oplus E^0(x) \oplus E^u(x)$ and for $a = s, u, 0$,
    \[
    d_x f E^a(x) = E^a(f(x));
    \]
    
    \item[(4)] for $v \in E^s(x)$, $w \in E^u(x)$, and  $n \ge 0$ we have
    \begin{equation*}
        ||d_xf^nv||\leq C(x)\lambda_1(x)^ne^{\varepsilon(x)n}||v||, ||d_xf^{-n}w||\leq C(x)\mu_1(x)^{-n}e^{\varepsilon(x)n}||v|| 
    \end{equation*}
    
    \item[(5)] for $v \in E^0(x)$
    \[
    C(x)^{-1}\lambda_2(x)^n e^{-\varepsilon(x)n}\|v\| \le \|d_x f^n v\| \le C(x)\mu_2(x)^n e^{\varepsilon(x)n}\|v\|, \quad n \ge 0,
    \]
    \[
    C(x)^{-1}\mu_2(x)^n e^{\varepsilon(x)n}\|v\| \le \|d_x f^n v\| \le C(x)\lambda_2(x)^n e^{-\varepsilon(x)n}\|v\|, \quad n \le 0;
    \]
    
    \item[(6)] the angles satisfy $\angle(E^a(x), E^b(x)) \ge K(x)$ where $a, b \in \{s, u, 0\}$ and $a \ne b$.
\end{itemize}
\end{definition}

\begin{remark}
    For an ergodic action $\alpha$ of $\Z^d$ on $(M,\mu)$, $\mu$ a Borel probability measure on $M$, from \cite[Proposition~2.6]{brhwzd}, it follows that for any $\vec n\in\Z^d$ the $\alpha(\vec n)$ invariant subset $\Lambda$ appearing in Theorem \ref{Higher rank osedelec} is a nonuniformly partially hyperbolic subset of $M$ with respect to  $\alpha(\vec n)$.
\end{remark}

For a nonuniformly partially hyperbolic set $Z\subseteq M$ with respect to $f$ we have the existence of stable manifolds. We have that $Z=\bigcup_{l\in\N}Z_l$, where each $Z_l$ is compact, and where the functions $Z\rightarrow [0,\infty]$ appearing in Definition \ref{definition of nonuniformly partially hyperbolic subset} are all continuous.  

\subsubsection{Local stable manifolds}

For each regular $Z^l$, we have the existence of stable manifolds.

\begin{theorem}\cite[Theorem~7.1, Section~7.3.5]{bpintro}\label{stable manifold theorem}
   For every $x \in Z_l$ there exists a local stable manifold $V^s(x)$ such that $x \in V^s(x)$, $T_x V^s(x) = E^s(x)$, and for every $y \in V^s(x)$ and $n \ge 0$,
\[
\rho(f^n(x), f^n(y)) \le T(x)(\lambda')^n \rho(x, y),
\]
where $\rho$ is the distance in $M$ induced by the Riemannian metric, $\lambda'$ is a number satisfying $0 < \lambda e^\varepsilon < \lambda' < 1$, and $T: \Lambda_\iota \to (0, \infty)$ is a Borel function satisfying
\[
T(f^m(x)) \le T(x)e^{10\varepsilon|m|}, \quad m \in \mathbb{Z}.
\]
Furthermore, $f^m(V^s(x)) \subset V^s(f^m(x))$ for every $m \in \mathbb{Z}$; and $\locstable(x)$ depends uniformly continuously on $x\in Z^l$.
\end{theorem}

\begin{remark}
    In fact, for every $x \in Z^l$,
\[
\mathcal{W}_{\text{loc}}^s(x) = \exp_x \{ (v, \psi_x^s(x)) : v \in B^s(r) \},
\]
where $B^s(r)$ is the ball of radius $r$, for some $r \geq r^l$, of the vector subspace $E^s(x)$ of $T_x M$, and $\psi_x^s : B^s(r) \to E^u(x)$ is a smooth map satisfying $\psi_x^s(0) = 0$ and $d_0 \psi_x^s = 0$.
\end{remark} 

\subsubsection{Absolute continuity along stable manifolds}\label{Absolute continuity along stable manifolds section}
 
Given $x\in\mathcal R^l$, consider the union 
\begin{equation}
    \mathcal{Q}^l(x)=\bigcup_{w\in\mathcal R^l\cap B^M_{r^l}(x)}\locstable(w)
\end{equation}
of the family of local stable manifolds 
\begin{equation}\label{family of local stable manifolds}
    \mathcal L^l(x)=\{\locstable(w):w\in\mathcal R^l\cap B^M_{r^l}(x)\}.
\end{equation}

A \textit{local smooth submanifold} of $M$ is the graph of a smooth injective function $\mathbb{R}^k \to M$. A \textit{local transversal} to the family $\mathcal{L}^l$ is a local smooth submanifold which is uniformly transverse to every manifold in this family, i.e., the angle between $T$ and $\mathcal{W}_{\text{loc}}^s(w)$ is uniformly away from zero for every $w\in\mathcal R^l\cap B^M_{r^l}(x)$. For any $T^1$ and $T^2$ local transversals to $\mathcal{L}^l(x)$, we have a corresponding holonomy map
\begin{equation}\label{holonomy between transversals along stable}
    \mathcal{Q}^l(x) \cap T^1 \to \mathcal{Q}^l(x) \cap T^2, 
\end{equation}
by setting $\locstable(w)\cap T^1= z_1 \mapsto z_2=\locstable(w)\cap T^2$, for every $w\in\mathcal R^l\cap B^M_{r^l}(x)$.

For such transverse submanifolds $T$ to the family $\mathcal L^l(x)$, we denote by $\nu_T$ the volume measure on $T$ inherited from the restriction of the Riemannian metric of $M$ to $W$.

\begin{theorem}\cite[Theorem~8.2, Remark~8.19]{bpintro} \label{absolute continuity along stable}
    If in (\ref{holonomy between transversals along stable}) $\nu_{T^1}(T^1\cap \mathcal{Q}^l(x))>0$, the holonomy map (\ref{holonomy between transversals along stable}) is absolutely continuous, with respect to the volume measures $\nu_{T_1}$ and $\nu_{T_2}$.
\end{theorem} 

\section{Harmonic measures, unique $P$-ergodicity and $G$-invariant measures}\label{Harmonic measures, unique $P$-ergodicity and $G$-invariant measures}

In this subsection we follow \cite{dk} and \cite{candel}. 

We first discuss the notion of harmonic measures for foliations of compact manifolds. We then study the relationship of these measures with  measures of actions of semisimple Lie groups invariant under Borel subgroups.

\subsubsection{Harmonic functions on Lie groups}
Let $G$ be a Lie group, and $K$ a compact subgroup, and $X$ is the symmetric space $G/K$. A function $f$ in $X$(or a left $K$-invariant function on $G$) is harmonic if $Df=0$ for all differential operators $D$ on $X$ which vanish all constant functions and it is left $G$-invariant, see \cite{godement}.Since this is a local property, it also makes sense on $G/\Gamma$. The following lemma gives a global sufficient condition for such functions. 

\begin{theorem}\cite[Main Theorem]{godement}\label{godement}
    A left $K$-invariant function $f$ on $G$ is harmonic if it is locally integrable and 
    \begin{equation*}
        f(g)=\int_Kf(xkg)dk,
    \end{equation*}
    for all $g,x\in G$.
\end{theorem}

\subsection{Transversely Invariant and Harmonic measures}

\subsubsection{Foliations and harmonic measures}
Suppose $\mathcal F$ is a foliation of a compact manifold $M$, with $C^3$ leaves, depending $C^1$ transversely. Then there exists a finite family of foliation boxes $D_i\times T_i$ covering $M$, where $T_i$ is transverse to the leaves. Recall that the change of coordinates from $D_i\times T_i$ to $D_j\times T_j$ are of the form
$$ (x_i,t_i)\mapsto (x_j=x_j(x_i,t_i), t_j(t_i)).$$
The maps $t_j(t_i)$ generate a pseudo-group on the union $T=\cup_i T_i$, 
called the holonomy pseudo-group. A measure on $T$ which is invariant under the holonomy pseudo-group is called a transversely invariant measure.

In \cite{garnett}, Lucy Garnett defined the notion of harmonic probability measures for $C^1$-transversely compact foliated spaces, such that each leaf is endowed with a Riemannian metric. In her work she showed these measures always exists, unlike holonomy invariant measures. In the setting of locally free actions of Lie groups on compact manifolds, it is not hard to see that invariant probability measures exist if and only if holonomy invariant measures exist, see \cite[Theorem~9.1]{plante} and Remark \ref{idea of stuck zimmer}.

Suppose $(M,\mathcal F)$ is a transversely $C^1$ and leafwise $C^3$ foliated space, endowed with  Riemannianm metrics along each leaf. Then there is a Laplace operator $\Delta$ associated to this foliation, which acts on measurable functions $f$ on $M$ which are twice differentiable along the leaves of the foliation by $\Delta f(x)=\Delta_L(f\upharpoonright L)(x)$, where $L$ is the leaf of $(M,\mathcal F)$ containing $x$, and $\Delta_L$ is the Laplacian of the Riemannian metric associated to $L$. 

Associated to the above Laplace operator we have for each $t\geq 0$ a diffusion operator $D^t:C^0(M)\rightarrow C^0(M)$ given by $D^t(f):=F(t,-)$, where $F$ is the solution of the heat equation 
\begin{equation*}
    \frac{\partial F}{\partial t}=\Delta F,
\end{equation*}
with initial condition $F(0,x)=f(x)$. 

If $m$ is a measure on $M$, we can consider for any  $t\geq 0$ the measure $D^tm$ given by 
\begin{equation*}
    \int f d(D^tm):=\int D^t(f)dm.
\end{equation*}

A probability measure $m$ on $M$ is \emph{harmonic}, if $D^tm=m$ for all $t\geq 0$. Such a measure is ergodic, if it cannot be written as a convex combination of other harmonic measures. 

The following result gives a local characterization of harmonic measures:

\begin{theorem}\cite[Theorem~1.c]{garnett}\cite[Proposition~5.2]{candel}\label{garnett}
    A measure $m$ is harmonic if and only if $m$ locally (in a distinguished foliation coordinate system) disintegrates into a transversal sum of leaf measures, where almost every leaf measure is a positive harmonic function times the Riemannian leaf measure
\end{theorem}

%A harmonic measure is \emph{ergodic} if it cannot be written as a convex sum of different harmonic measures. As we can think of foliations of a Riemannian manifold as ``dynamical systems'' by considering the diffusion semigroup coming from the Brownian motion on each leaf, for an ergodic measure $m$ on a codimension one foliated compact manifold $(M,\mathcal F)$, the quantity
%\begin{equation*}
 %   \lim_{t\leftarrow\infty}\frac 1t\log |Dh_{\gamma|[0,t]}|
%\end{equation*}
%always exists and it is equal to a fixed real number $\lambda$, for all $x\in M$, and for $W_x$-a.e. $\gamma\in\Gamma_x$, where $\Gamma_x$ is the set of all continuous paths starting at $x$ and contained in the leaf passing through $x$, $W_x$ is the Wiener measure associated to the Brownian motion on this leaf associated to the metric on it, and $|\cdot |$ is any transverse metric on $M$ to the foliation $\mathcal F$. We call $\lambda$ the \emph{Lyapunov exponent}, associated to the codimension one foliation $(M,\mathcal F)$. See \cite{dk}.

The following result of Deroin and Kleptsyn regarding the existence of \emph{transversely invariant} probability measures, that is, a measure invariant under the holonomy pseudo-group,  for codimension one foliations of compact manifolds, is  essential for this paper.

Let $(M,\mathcal F)$ be a codimension one foliation of a closed manifold,  $C^3$ leafwise, and  $C^1$  transversely. 

\begin{theorem}\cite[Theorem~1.1.d]{dk}\label{dk}
If $(M,\mathcal F)$ has no holonomy invariant probability measure, then there exist $\m_1,\ldots, \m_r$ minimal subsets of $(M,\mathcal F)$ such that 
\begin{enumerate}
    \item $(\m_i,\mathcal F|\m_i)$ has a unique harmonic measure $\mu_i$;
    \item When $t$ goes to infinity, the diffusions $D^t f$ of a continuous
function $f:M\rightarrow\R$ converge uniformly to the function
$\sum_j c_j p_j$, where $c_j = \int f d\mu_j$. 
\end{enumerate}
\end{theorem}

This implies the following.

\begin{corollary}
    The measures $\mu_1,\ldots,\mu_r$ are the extremal points of the convex set of harmonic probability measures on $M$.
\end{corollary}

\begin{proof}
Let $m$ be any harmonic probability measure on $M$. Then 
\begin{equation*}
    \int D^t(f)dm=\int fd(Dtm)=\int fdm,
\end{equation*}
so that by (3) in the previous Theorem it follows that 
\begin{equation*}
    \int fdm=\lim_{t\rightarrow\infty}\int D^tfdm=\int \lim_{t\rightarrow\infty}D^tfdm=\sum_j c_j \int p_jdm=\sum_jh_j\int fd\mu_j,
\end{equation*}
where we can exchange the limit and the integral by uniform convergence, and $h_j=\int p_jdm$, thus $m=\sum_jh_j\mu_j$. Therefore, the convex hull of the measures $\mu_1,\ldots,\mu_r$ is the set of harmonic probability measures on $M$; that they are extremal points of this set follows from (1) in the previous Theorem.
\end{proof}

\subsection{Harmonic measures and $P$-invariant measures}

Let $G$ be a semisimple Lie group with finite center; $K$ a maximal subgroup of it; and $P$ a minimal parabolic of it. Let $\mu$ be an absolutely continuous probability measure on $G$ with $C^\infty$ density. 

From Theorem \ref{dk} we prove a dichotomy between the existence of invariant measures for locally free actions of $G$ on compact manifolds $M^{\dim(G)+1}$, and having finitely many extremal points for the set of $P$-invariant probability measures on $M$. 

To do this, we show that for any absolutely continuous and $K$-bi-invariant probability measure $\mu$ on $G$, any $\mu$-stationary probability measure is in fact harmonic with respect to the foliation induced by the $G$-orbits of the action, by using classic work of Furstenberg \cite{furstenberg}. This fact appears in \cite{candelpres}, without proof.

A Borel probability measure $\nu$ on $M$ is called \emph{stationary} if $\mu\ast\nu:=\int g_\ast\nu d\mu(g)=\nu$. We denote by $\mathcal M^1_\mu(M)$ the set of all such measures; and the set of $P$-invariant probability measures on $M$ by $\mathcal M^1_P(M)$.

We have the following classical theorems of Furstenberg.

\begin{theorem}\cite[Theorem~2.3]{furstenberg2}\cite[Proposition~1.3]{nzinventiones}\label{statboundary}
    Let $Q\leq G$ be a parabolic subgroup. Then $G/Q$ has a unique $\mu$-stationary measure. This measure is in the class of the smooth measure. If $\mu$ is left $K$-invariant, then this measure is the unique $K$-invariant measure on $G/Q$.
\end{theorem}

\begin{theorem}\cite[Theorem~2.1]{furstenberg2}\cite[Theorem~1.4.2]{nzinventiones}\label{bijectionPmu}
    Let $G$ be a connected non-compact semisimple Lie group with finite center. Let $\tilde\nu_0$ be any probability measure on $G$ which under the canonical $p:G\rightarrow G/P$ satisfies $p_\ast\tilde\nu_0=\nu_0$, where $\nu_0$ is the unique $\mu$-stationary measure on $G/P$. Then the assignment $\lambda\mapsto\tilde\nu_0\ast\lambda$ is a convex bijection from $\mathcal M_P^1(M)$ onto $\mathcal M_\mu^1(M)$.
\end{theorem}

The following fact appears without proof in \cite{candelpres}.

\begin{proposition}\label{stationaryharmonic}
    Any probability measure on $M$ which is $\mu$-stationary, is harmonic for the foliation of $M$ induced by the $G$-orbits.
\end{proposition}

\begin{proof}
Let $\nu$ be a $\mu$-stationary probability measure. 

Recall that since $G\acts M^{\dim(G)+1}$ is locally free, its leaves yield a codimension one foliation of $M$, which we denote by $\mathcal F_G$. 

Since  $\nu=\mu\ast\nu$  and since $\mu$ is absolutely continuous on $G$ with smooth densitity, we get that the leafwise measures of $\nu$ with respect to the action $G$, denoted by $[\nu_x^G]$, are absolutely continuous with respect to the volume measure of the orbits of the action $G\acts M$, with smooth density.  

Let $B\simeq U\times [-1,1]$ be any foliation box of  $(M,\mathcal F_G)$, where $U\simeq\R^{\dim G}$. By identifying the plaques of $B$ and using Rokhlin's disintegration, the previous paragraph implies that the restriction of the measure $\nu$ to $B$ disintegrates as
    \begin{equation*}
        \nu|_B=\int h_t d\vol_{L_t}d\gamma(t),
    \end{equation*}
    where $\gamma$ is a finite measure supported on $[-1,1]$, $L_t$ is the $G$-orbit associated to the plaque corresponding to $t$, for $\gamma$-a.e. $t$, and $h_t$ is a non-negative function supported on the plaque associated to $t$, for $\gamma$-a.e. $t$. As foliation boxes intersect in open sets, $M$ can be covered by a finite  number of such boxes, and because of the uniqueness statements in both the Radon-Nikodym and the Rokhlin disintegration theorems, it follows that  we can extend each $h_t$ to a global function on the whole $G$-orbit containing the plaque associated to $t$, for $\gamma$-a.e. $t$. Denote this global function by $\overline{h_t}$.

    Let $\eta$ be the Haar measure on $K$ of total measure $1$, which can be seen as a measure on $G$, as $K$ is in particular a closed subset of $G$. It satisfies $\eta\ast\eta=\eta$.  If $p:G\rightarrow G/P$ is the canonical projection, then $p_\ast\eta$ is the $\mu$-stationary probability measure on $G/P$, by Theorem \ref{statboundary}. Applying Theorem \ref{bijectionPmu}, there exists a $P$-invariant measure $\lambda$ on $M$ such that $\nu=\eta\ast\lambda$. 

    Since by hypothesis we have $\mu\ast\nu=\nu$, we get from the construction of $\overline{h_t}$ that 
    \begin{equation}\label{muharmonicity}
        \overline{h_t}(x)=\int\overline{h_t}(gx)d\mu(g)
    \end{equation}
    for all $x\in L_t$, and $\gamma$-a.e. $t$. Let $\hat{h_t}$ be the lift of $\overline{h_t}$ to $G$, by taking a $G$-equivariant diffeomorphism $G/\Gamma_t\rightarrow L_t$, $\Gamma_t\leq G$ discrete subgroup, i.e., $\hat{h_t}(g)=\overline{h_t}(g\Gamma_t)$ for all $g\in G$. Then from (\ref{muharmonicity}) we get that 
    \begin{equation}\label{liftmuharmonicity}
        \hat{h_t}(h)=\int\hat{h_t}(gh)d\mu(g)
    \end{equation}
    for all $h\in G$.

    The argument in the first paragraph of the proof in \cite[theorem~5.6]{furstenberg} shows that the function
    \begin{equation}\label{jhnvmvkf3}
        g\mapsto\int\hat{h_t}(gkg_1)d\eta(k)
    \end{equation}
    is constant, for all $g_1\in G$, using (\ref{liftmuharmonicity}). For this, we may assume, without loss of generality, that $\mu$ is of class $B_\infty$, as defined in \cite[Section~5.5]{furstenberg}.

    We also have that 
    \begin{equation}\label{mbknmjegh}
        \int \hat h_t(kg_1)dk=\hat h_t(g_1),
    \end{equation}
    for all $g_1\in G$. This follows from the fact that $\eta\ast\nu=\nu$, as $\nu=\eta\ast\lambda$ and $\eta\ast\eta=\eta$.

    From (\ref{jhnvmvkf3}) and (\ref{mbknmjegh}) we obtain that 
    \begin{equation}\label{fhlhkhf}
        \int \hat h_t(gkg_1)dk=\hat h_t(g_1)
    \end{equation}
    for all $g,g_1\in G$, and $\gamma$-a.e. $t$. 
    
    Finally, as $\eta$ is the Haar measure on $K$, and $K$ is unimodular, as it is compact, we get that $k_\ast\eta=\eta$, and this in turn implies $\hat h_t$ is left $K$-invariant. Hence, from Theorem \ref{godement} and , we get that for $\gamma$-a.e. $t$, $\hat h_t$ is left $K$-invariant. This, along with (\ref{fhlhkhf}) and Theorem \ref{godement}, implies that $\hat h_t$ is annihilated by all differential operators on $G$ which are right $G$-invariant and annihilate all constant functions on $G$; since this is a local property, the same is true for $h_t$, for $\gamma$-a.e. $t$, i.e., $h_t$ is harmonic, and we are done by Theorem \ref{garnett}.
\end{proof}

Together, Proposition \ref{stationaryharmonic} and Theorem \ref{dk}
have the following consequence.

\begin{proposition}\label{uniquePergodicity}
    Let $G$ be a connected, non-compact, semisimple Lie group with finite center, having property (T), and acting locally freely and in a $C^1$ manner on a closed manifold $M$ of dimension $\dim(G)+1$. Let $P\leq G$ be a minimal parabolic subgroup. Suppose $G\acts M$ has no invariant probability measure. Then, there are minimal subsets $\mathcal M_1,\ldots, \mathcal M_r$ of the action $G\acts M$ such that:
    \begin{enumerate}
        \item The restricted action $P\acts \mathcal M_j$ is uniquely ergodic, for all $j=1,\ldots, r$, denote by $\mu_j$ the corresponding $P$-invariant probability measure on $\mathcal M_j$;
        \item The set of $P$-invariant probability measures on $M$ is the convex hull of the measures $\{\mu_j\}_{j=1}^r$, and these measures are the extremal points of this set.
    \end{enumerate}
\end{proposition}
\begin{proof}
    Suppose $\m$ is a minimal subset of $G\acts M$ with no invariant measure. Then $\m$ does not have a holonomy invariant measure, for otherwise $M$ would have a holonomy invariant measure supported on $\m$, since $\m$ is closed; and then, by the Remark \ref{idea of stuck zimmer}, $M$ would have a finite cover equivariantly diffeomorphic to $G/\Gamma\times S^1$, such a cover has a $G$-invariant measure, as all its orbits are closed, and from this measure a transversely invariant measure can be built, see \cite[Theorem~9.1]{plante} and \cite[Theorem~5.1]{sz}, yielding a contradiction.
    
    We can apply Theorem \ref{dk} in this case, as the orbits of the action $G\acts M$ carry an analytic differentiable structure, that depends transversely in a $C^1$ manner, coming from $G$. From this theorem we get that the foliation induced by the locally free action $G\acts \m$ has a unique harmonic measure. Let $\lambda$ be a $P$-invariant probability measure on $\m$, which exists as $P$ is amenable and $\m$ is compact. $\lambda$ is also a probability measure on $M$. Then from theorems \ref{bijectionPmu} and \ref{stationaryharmonic} we get that $\tilde\nu\ast\lambda$ is harmonic on $M$, recall $\tilde\nu$ is the extension to $G$ of the normalized Haar measure of $K$. However, as $\lambda$ is supported on $\m$, we get from the definition of convolutions that $\tilde\nu\ast\lambda$ is also supported on $\m$, and it is a harmonic probability measure. Thus, because of theorem \ref{bijectionPmu}, we get that $\lambda$ is the unique $P$-invariant probability measure on $\m$.
\end{proof}

\section{Topological measure rigidity}\label{section topological measure rigidity}

Let $G$ be a non-compact split simple Lie group with finite center. Throughout this section we consider a locally free action of $G$  on a closed manifold $M$.

For the rest of this section, fix a maximal torus $A\leq G$, along with a minimal parabolic subgroup $P\leq G$ such that $A\leq P$. We denote by $\Pi_P$ the set of all simple positive roots, with respect to $(G,A)$, of $P$. 

Throughout this section, fix a locally free action $G\acts M$, on a closed manifold $M$. In this section we give conditions under which a $P$-invariant probability measure on $M$ is $G^{-\alpha}$-invariant, for $\alpha\in\Pi_P$. 

\subsubsection{Some Lie group preliminaries}\label{Some Lie group preliminaries}

Fix a maximal compact subgroup $K$ of $G$. We fix a right $G$-invariant, bi-$K$-invariant metric $d_G$ on $G$. For a semisimple element $a$ of $G$ we consider the following subgroups of $G$:
\begin{equation}\label{Na and La}
\begin{split}
    N_a:=\{g\in G:\lim_{n\rightarrow-\infty} d_G(a^nga^{-n},e)=0\}\\
    L_a:=\{g\in G:\lim_{n\rightarrow\infty} d_G(a^nga^{-n},e)=0\}.
\end{split}
\end{equation}
We denote by $C_a$ the center of $a$ in $G$. If $a\neq e$, denote by $Q_a$ the proper parabolic subgroup of $G$ generated by $P$ and $C_a$.

For any closed subgroup $H$ of $G$ and any $\delta>0$, we set 
\begin{equation*}
    H^{<\delta}=\{h\in H:d_G(h,e)<\delta\}.
\end{equation*}

\begin{remark}
     $C_a$ normalizes both $N_a$ and $L_a$.
\end{remark}

\begin{remark}
    Since the given action $G\acts M$ is locally free and $M$ is compact, there exists some $\delta>0$ such that for any $n\in N_a\setminus\{e\}$ we have $nx\neq x$. Indeed,  if $n\in N_a$ and $x\in M$ with $nx=x$, then for some large $k\in\N$ we have that $d(e,a^kna^{-k})<\delta$, however $a^kx=a^knx=a^kna^{-k}a^kx$, which by our choice of $\delta$ implies that $a^kna^{-k}=e$, i.e., $n=e$. Therefore, for any $x\in M$ the map $N_a\ni n\mapsto nx$ is injective. The same holds for $L_a$.
\end{remark}

\begin{lemma}\label{multiplication maps}
    The maps $C_a\times N_a\rightarrow Q_a$ and $L_a\times C_a\times N_a\rightarrow G$ given by $(c,n)\mapsto cn$ and $(l,c,n)\mapsto lcn$ are diffeomorphisms of a neighborhood of the identity of $C_a\times N_a$ and $L_a\times C_a\times N_a$, respectively, onto a neighborhood of the identity of $Q_a$ and $G$, respectively
\end{lemma}
\begin{proof}
    This follows by showing the Jacobian of each of the maps above has full rank, as $\lie(G)=\lie(L_a)\oplus\lie(C_a)\oplus\lie(N_a)$ and $\lie(Q_a)=\lie(C_a)\oplus\lie(N_a)$, by using the root space decomposition of $\lie(G)$ with respect to $A$.
\end{proof}

\subsubsection{Deodhar's lemma on Weyl wall reflections}

Given a root $\alpha\in A^\ast$, the reflection in $A$ of the hyperplane $\ker(\alpha)$, which we denote by $w_\alpha$, is uniquely represented by an element of the Weyl group $N_G(A)/C_G(A)$ of $G$(here $N_G(A)$ denotes the normalizer of $A$ in $G$). We can get representatives in $N_G(A)$ of such reflections that can be expressed as a product of specific unitary elements. More precisely we have the following lemma by Deodhar, which uses the split hypothesis on $G$.

\begin{lemma}\cite[Lemma~1.3]{deodhar}\label{deodhar}
    Given any $u\in U^{-\alpha}\setminus\{e\}$, there are $u_1\in U^{-\alpha}$ and $v_1\in U^\alpha$ such that the product $w=u_1v_1u$ belongs to $N_G(A)$; and acts on $A$ as $w_\alpha$, i.e., $w_\alpha w^{-1}\in C_G(A)$.
\end{lemma}

\subsubsection{Birkhoff ergodic theorem and measures invariant under unstable subgroup}

We need the following observation about $s$-ergodic measures that are also invariant under the unstable subgroup of $s$ in $G$. See also \cite[Proposition~8.4]{dh}.

\begin{lemma}\label{birkhoff for almost all px}
    Let $a\in A$, and suppose $\mu$ is a $N_a$-invariant and $a$-ergodic probability measure on $M$. For $\mu$-a.e. $x$ and a.e. $n\in N_a$, the Birkhoff ergodic theorem holds for $nx$ with respect to $\mu$ and $a$.  
\end{lemma}

\begin{proof}  
Since $N_a$ is by construction the unstable subgroup of $a$ in $G$, we obtain from proposition \ref{haarconditionals} that the Birkhoff ergodic theorem holds for $ux$ with respect to $\mu$ and $a$, for $\mu$-a.e. $x$ and $\haar_{N_a}$-a.e. $u$, since this theorem holds for $\mu$-a.e. $x$. %If we take any such $u$ and $x$, along with any $c \in C_a$, we get from this that for any $f \in C^0(M)$,  
%\[
%\lim_{N \to \infty} \frac{1}{N} \sum_{j=0}^{N-1} f( a^j cu x)
%= \lim_{N \to \infty} \frac{1}{N} \sum_{j=0}^{N-1} f(c a^j u x)
%= \mu(f\circ L_c),
%\]
%where $L_c$ is the diffeomorphism $M \to M$ given by left multiplication by $c$. However, since $\mu$ is $P$-invariant, it follows that the latter integral equals $\mu(f)$. Therefore, we obtain our lemma because of Fubini's theorem along with the Langlands decomposition of $P$.
\end{proof}

\subsubsection{$a$-ergodic measures and transverse Lyapunov exponents}

Take $a\in A$. Any $a$-ergodic probability measure on $M$ has a corresponding Osedelets' decomposition of $TM$, on a subset of full measure of $M$, along with corresponding Lyapunov exponents. Notice that we have canonical Lyapunov exponents, i.e. roots of the maximal torus containing $a$, coming from the Cartan flow of the locally free action $G\acts M$. 

\begin{definition}
    Given an $a$-ergodic probability measure $\nu$ on $M$, we consider the set $\Lambda_\nu^T(a)$ of \emph{transverse Lyuapunov exponents} of $\nu$ to be the set, consisting of real numbers, of all Lyapunov exponents of $a$ with respect to $\nu$ whose corresponding associated Lyapunov subspaces of $TM$ are not tangent to the orbits of $G\acts M$.
\end{definition}

%\begin{definition}
 %   Since $G\acts M$ is locally free, and its orbits are of codimension $1$, we obtain that for any $P$-ergodic probability measure $\mu$ on $M$ there is a unique Lyapunov exponent, which we shall denote by $\lambda_\mu^T\in A^\ast$ and call \emph{transversal Lyapunov exponent}, with respect to the $A$-ergodic measure $\mu$, such that its corresponding Osedelets' distribution is not contained in the subbundle $TG$ of $G$-orbits of $TM$.
%\end{definition}

We prove that there is no positive entropy for $\mu^P$ along the $G^\beta$-orbit partition for all $\beta\in\Delta_{Q_0}^c$. Compare this with \cite[Section~5.3]{brhw}; as in this reference, we will rely on the high entropy method, see \cite{ek} and \cite{el}.

\subsubsection{The high entropy method}
Let $G$ be a semisimple split Lie group, $A\leq G$ a maximal subtorus, and denote by $\Delta:=\Delta_{(G,A)}$ the corresponding root system.

The following proposition follows by combining the decomposition theorem in \cite[Theorem~8.4]{ek} with the invariance argument appearing in \cite[Proposition~7.1]{ekcomm}, using that there are points in the interiors of all the Weyl chambers of $A$ for which $\mu$ is ergodic, because of theorem \ref{pughshub}. In the following statement we make no assumption on the dimension of $M$.%!!!!!!!!!!!!!!!!!!!!!!!!!!!!!!!!!!!!!!!!!!!!!!!!

\begin{proposition}\label{high entropy method}
    Suppose we have a locally freely action $G\acts M$, $M$ closed. Let $\mu$ be an $A$-invariant and ergodic measure on $M$. Let $\alpha,\beta\in\Delta$ be such that $\alpha+\beta\in\Delta$. If for $\mu$-a.e. $x$ both leafwise measures $\mu_x^{G^\alpha}$ and $\mu_x^{G^\beta}$ are non-atomic, then $\mu$ is $G^{\alpha+\beta}$-invariant.
\end{proposition}

\subsection{Topological measure rigidity}

The main result of this section is the following proposition, which is heavily inspired by the work of Katok and Spatzier, and Deroin and Hurtado, see \cite[Corollary~5.2]{katokspatzier} and \cite[Lemma~8.2]{dh}. Let $\alpha$ be a root of $(G,A)$.

\begin{proposition}[Topological Measure Rigidity]\label{topological measure rigidity}
Suppose $s \in \ker(\alpha) \setminus \{e\}$. Let $\mu$ be an $s$-ergodic and $C_G(A)\ltimes\langle N_s,G^{\alpha}\rangle$-invariant probability measure on $M$ such that all elements in $\Lambda^T_\mu(s)$ are negative. For $\mu$-a.e. $x$ we have that if $ux\in\supp(\mu)$ for some $u\in G^{-\alpha}\setminus\{e\}$, then $\mu$ is $G^{-\alpha}$-invariant. 
\end{proposition}

The main goal of this section is to prove this proposition.  Compare this result with the aforementioned facts,  \cite[Corollary~5.2]{katokspatzier} and \cite[Lemma~8.2]{dh}. In particular, the proof of proposition \ref{topological measure rigidity} is motivated by the one of the latter paper which is itself a modification of the Hopf argument and uses accessibility coming from a factor map of a suspension space onto $ G/\Gamma$, for $\Gamma$ a lattice, by exploiting the ergodicity of singular elements of $A$ on $G/\Gamma$, with the Haar measure. Here, we use Proposition \ref{deodhar} coming from \cite{deodhar} to get accessibility. The idea of using \cite{deodhar} to get accessibility appears for instance in \cite{damjanovic2025zimmerprogrampartiallyhyperbolic}.

Unlike in proposition \ref{topological measure rigidity}, for the following proposition we do not need any hypothesis on the transverse Lyapunov exponents.

The following fact is a direct consequence of proposition \ref{topological measure rigidity} under the hypothesis of this fact, which justifies its name; compare the statement below with \cite[Corollary~5.2]{katokspatzier}. However, we shall prove this statement in a more general setting, more precisely, we do not make the assumption that the transverse Lyapunov exponents associated to the ergodic measure are negative. 

\begin{proposition}[Entropic measure rigidity]\label{entropic measure rigidity}
    Let $s\in\ker(\alpha)\setminus\{e\}$  and suppose $\mu$ is an $s$-ergodic and $C_G(A)\ltimes\langle N_s,G^{\alpha}\rangle$-invariant probability measure on $M$. If  $\mu^{G^{-\alpha}}_x$ is non-atomic for $\mu$-a.e. $x$, then $\mu$ is $G^{-\alpha}$-invariant. 
\end{proposition}
\begin{proof of entropic measure rigidity negative transverse exponents}
    Let $x$ be generic enough for $\mu$ so that for $\mu_x^{G^{-\alpha}}$-a.e. $u$ we have $ux\in\supp(\mu)$; and $\mu_x^{G^{-\alpha}}$ is nonatomic, which can be done by hypothesis. Then there exists $u\in G^{-\alpha}\setminus\{e\}$ such that $ux\in\supp(\mu)$. Therefore, by proposition \ref{topological measure rigidity} we get that $\mu$ is $G^{-\alpha}$-invariant.
\end{proof of entropic measure rigidity negative transverse exponents}
\subsection{Entropy measure rigidity}

%Let us first show that the only possible $\alpha\in\Delta\setminus\Delta_{Q_0}$ such that for $\mu$-a.e. $x$ the leafwise measure  $(\mu^P)_x^{G^\alpha}$ is nonatomic are such that $-\alpha$ is a simple positive root with respect to $P$.

%\begin{lemma}\label{unique possible non rigidity}
 %   If $\alpha\in\Delta\setminus\Delta_{Q_0}$ and $(\mu^P)_x^{G^\alpha}$ is nonatomic, then $-\alpha\in\Pi_P$.
%\end{lemma}

%\begin{proof}
%Since $\Delta_P\subseteq\Delta_{Q_0}$, it follows that $-\alpha\in\Delta_P$. 

%If $-\alpha \notin \Pi_P$, then there exists some $\beta \in \Pi_P$ such that $\beta$ is a summand of $-\alpha$, and $-\beta \notin \Delta_Q$. Such a simple root exists, for otherwise all the negations of the summands of $-\alpha$ in $\Pi_P$ would belong to $\Delta_Q$, and since $\Delta_Q$ is closed under brackets, it would follow that $\alpha \in \Delta_Q$, contradicting our hypothesis. As $-\alpha \in\Delta_P\setminus\Pi_P$, we get that $-\alpha = \beta_0 + \dots + \beta_k, $ for some $\beta_i \in \Pi_P$, $k\geq 1$, with $\beta_0=\beta$. Then $\gamma:=\beta_1+\cdots+\beta_r$ is a root in $\Delta_P^+$, and we would get that $\mu^P$ would be $G^{\beta}$ invariant by applying the high entropy method, proposition \ref{high entropy method}, since $G^{\beta} = [G^{-\alpha}, G^\gamma]$, and as $\mu^P$ is $G^\gamma$ invariant. This is a contradiction. Therefore $-\alpha\in\Pi_P$.

%\end{proof}

We now give the general proof of proposition \ref{entropic measure rigidity}, without using proposition \ref{entropic measure rigidity}. Although we will not use this general version in the proof of the main result of this paper, we give a proof of the general case, as some of the ideas appearing in the proof of Proposition \ref{topological measure rigidity} will appear in the following argument.

\begin{proof of entropic measure rigidity} 
    We just need to show that 
    \begin{equation*}
        (w_\alpha)_\ast\mu=\mu,
    \end{equation*}
    since $(w_\alpha)_\ast\mu$ is $G^{-\alpha}$-invariant, as $\mu$ is $G^{\alpha}$-invariant by hypothesis. 

    Since $s$ and $w_\alpha$ commute, we get that $(w_\alpha)_\ast\mu$ is $s$-ergodic as well.

    For $\mu$-a.e. $x$ we have the following:
    \begin{enumerate}
        \item\label{jhkmsvjf} the Birkoff ergodic theorem holds for $u'uvx$ with respect to $s$ and $\mu$ for $\mu_x^{G^{-\alpha}}$-a.e. $v$, $\haar_{G^\alpha}$-a.e. $u$, and $\haar_{N_s}$-a.e. $u'$; 
        \item\label{dbgkglhk} the Birkhoff ergodic theorem holds for $u'ux_0$ with respect to $s$ and $\mu$ for $\haar_{N_s}$-a.e. $u'$ and $\haar_{G^\alpha}$-a.e. $u$.
    \end{enumerate}
    Both properties hold as $\mu$ is $s$-ergodic and $\langle N_s, G^\alpha\rangle$-invariant.

    This can be done since $\mu$ is $s$-ergodic and $P$-invariant. Fix an $x_0$ along with some $v_0\in G^{-\alpha}\setminus\{e\}$ satisfying property (\ref{dbgkglhk}) above, which can be done as $\mu_x^{G^{-\alpha}}$ is non-atomic.

    By proposition \ref{deodhar} we get that there are $u_1\in G^{\alpha}$ and $v_1\in G^{-\alpha}$ such that $w_\alpha=v_1u_1v_0c$ for some $c\in C_G(A)$; however, $\mu$ is invariant under $c$, thus, to prove $w_\alpha\mu=\mu$, we need only prove $(v_1u_1v_0)_\ast\mu=\mu$, since
    \begin{equation*}\label{reduction of measure}
        (w_\alpha)_\ast\mu=(v_1u_1v_0c)_\ast\mu=(v_1u_1v_0)_\ast\mu.
    \end{equation*}

    Pick any $f\in C^0(M)$ and $\epsilon>0$. Since $M$ is compact, $f$ is uniformly continuous on $M$, hence there exists $\delta_0>0$ such that $|f(gx)-f(x)|<\epsilon$ whenever $g\in G^{<\delta_0}$. 

    Using the second map indicated in lemma \ref{multiplication maps} for $s$, and property (\ref{dbgkglhk}) above, we get that  for all $g\in G^{\leq\delta_0}$ we have 
    \begin{equation}\label{afba45hoh}
        \left |\limsup_{N\rightarrow\infty}\frac1N\sum_{j=0}^{N-1}f(s^jgu_1v_0y)-\int fd\mu\right |<\epsilon,
    \end{equation}
    as any $g\in G^{<\delta_0}$ can be written as $lcu$ for some $l\in L_s$, $c\in C_s^{\leq\delta_0}$ and $u\in N_s$ and because of our choice of $\delta_0$. 
    
    As $(w_\alpha)_\ast\mu=(v_1u_1v_0)_\ast\mu$ we get that the Birkhoff ergodic theorem holds for $u'vv_1u_1v_0 x_0$ with respect to $s$ and $(w_\alpha)_\ast\mu$ for a.e. $u'\in N_s$ and a.e. $v\in G^{-\alpha}$, because of our choice of $x_0$ with respect to property (\ref{jhkmsvjf}), and since $w_\alpha G^{-\alpha}w_\alpha^{-1}=G^\alpha$. Thus, for $\haar_{G^{-\alpha}}$-a.e. $v\in G^{-\alpha}$ and $\haar_{G}$-a.e. $g\in G^{<\delta_0}$ we have 
    \begin{equation}\label{asfoibv2}
        \left |\limsup_{N\rightarrow\infty}\frac1N\sum_{j=0}^{N-1}f(s^jgvv_1u_1v_0x_0)-\int fd(w_\alpha)_\ast\mu\right |<\epsilon,
    \end{equation}
    by decomposing elements in $G^{<\delta_0}$ using lemma \ref{multiplication maps}, and applying Fubini's Theorem.
    
    We can pick $v$ arbitrarily close to $v_1^{-1}$ so that this property holds for a.e. $u'\in N_s$, and $vv_1u_1v_0x_0$ is close enough to $u_1v_0x_0$ in such a way  that $G^{<\delta_0}\cdot vv_1u_1v_0x_0$ intersects $G^{<\delta_0}\cdot u_1v_0x_0$, and this intersection is open in the immersed topology of the orbit $G\cdot x_0$. Hence, we can take a point in this intersection that satisfies both inequalities (\ref{afba45hoh}) and (\ref{asfoibv2}), this implies 
    \begin{equation*}
        \left |\int fd(w_\alpha)_\ast\mu-\int fd\mu\right |<2\epsilon,
    \end{equation*}
    and since $\epsilon>0$ and $f\in C^0(M)$ were chosen arbitrarily, this proves $(w_\alpha)_\ast\mu=\mu$.
\end{proof of entropic measure rigidity}

\subsection{Extension of a nonuniform partially hyperbolic system by the center}

Let $a\in A$ and $\mu$ be an $a$-ergodic probability measure on $M$. For such measures the multiplicative ergodic theorem implies the  existence of a nonuniform partially hyperbolic subset $Z$, with respect to $a$ seen as a diffeomorphism of $M$, of full $\mu$ measure, see \cite[Remark~5.9, Section~5.4]{bpintro}. In the following lemma we prove that such subsets can be extended along orbits of the centralizer of $a$ to obtain a nonuniform partially hyperbolic subset with respect to $a$.

\begin{lemma}\label{Ca closedness of R}
    Take any $a\in A\setminus\{e\}$. Let $\mu$ be an $a$-ergodic probability measure on $M$. Then there exists a nonuniform partially hyperbolic subset $\mathcal R$ with respect to $a$, of full $\mu$ measure, which is invariant under multiplication by elements of $C_a$. Furthermore, for all $x\in\mathcal{R}$ we have 
    \begin{equation}\label{distros of extended system}
        D_xcE^\sigma(x)=E^\sigma(cx)
    \end{equation}
    for all $c\in C_a$, and $\sigma\in\{s,u,0\}$; and
    \begin{equation}\label{stable manifolds of extended system}
        c\cdot\mathcal W^\sigma(x)=\mathcal W^\sigma(cx),
    \end{equation}
    for $\sigma\in\{s,u,cu\}$.
\end{lemma}
\begin{proof}
    By ergodicity of $\mu$, we can choose an $a$-invariant nonuniform partially hyperbolic subset $\mathcal S$, such that the functions $\lambda_1,\lambda_2,\mu_1,\mu_2,\varepsilon:\mathcal S\rightarrow (0,\infty)$ appearing in the definition of nonuniform partially hyperbolicity in \cite[Section~5.4]{bpintro} are all constant. 

    Consider the set
    \begin{equation*}
        C_a\cdot\mathcal S=\{c\cdot x: c\in C_a, x\in\mathcal S\}.
    \end{equation*}

    If we define the corresponding subspaces $E^\sigma(cx)$ for $\sigma\in\{s,u,0\}$ by the formula indicated in the lemma, it follows that these subspaces are $a$-equivariant as $c\in C_a$.

    By using the corresponding functions $C,K:\mathcal S\rightarrow(0,\infty)$ for the nonuniform partially hyperbolic subset $\mathcal S$ as in \cite[Section~5.4]{bpintro}, we can construct corresponding measurable functions $C,K:C_a\cdot\mathcal S\rightarrow(0,\infty)$, which satisfy the properties indicated in the definition given in \cite[Section~5.4]{bpintro}, with respect to the positive real numbers $\lambda_1,\lambda_2,\mu_1,\mu_2,\varepsilon$ indicated above.

    Finally, the last equality of the lemma follows from compactness of $M$ and the definition of unstable manifold.
\end{proof}

\begin{remark}
    If $\mu$ and $s$ satisfy the hypothesis of proposition \ref{topological measure rigidity}, then, by considering the Cartan flow coming from the action $G\acts M$, we obtain that for $\mu$-a.e. $x$ the stable manifold $\mathcal W^s(x,s)$ exists, then for such $x$ the center unstable manifold with respect to $s$ at $x$ exists, in fact $\mathcal W^{cu}(x,s)=Q_s\cdot x$, where $Q_s$ is the subgroup of $G$ generated by $N_s$ and $C_s$, which follows from the assumption that all elements in $\Lambda_\mu^T(s)$ are negative.
\end{remark}

\begin{remark}\label{especific center unstable}
    By using the last equality of lemma \ref{Ca closedness of R} and the previous remark we get that since $Q_s=\langle N_s, C_s\rangle$, then for $\mu$-a.e. $x$ and $\haar_{Q_s}$-a.e. $q\in Q_s$ the central unstable manifold $\mathcal W^{cu}(qx;s)$ exists and equals $Q_s\cdot qx$.
\end{remark}

\subsection{Almost everywhere continuity of Birkhoff averages}

The following fact yields a ``continuity property" for the class of measures appearing in proposition \ref{topological measure rigidity}.

Let $s\in A$, and let $\mu$ be an $s$-ergodic, $\langle s\rangle\ltimes N_s$-invariant probability measure on $M$.

%Denote by $Q_{s,\alpha}$ the closed subgroup of $G$ equal to $\langle N_s,G^\alpha\rangle$.
\begin{proposition}\label{local rigidty of Pinvariant measures}
     For $\mu$-a.e. $x$ we have that for any, $f\in C^0(M)$, $\eta>0$ and $\epsilon>0$, there exists  $\delta=\delta(\mu,x,f,\eta,\epsilon)>0$ such that for any  $y$ with $d(x,y)<\delta$, the set
    \begin{equation*}
        \left\{q\in Q_{s}^{<\eta}: \left |\limsup_{N\rightarrow\infty}\frac1N\sum_{j=0}^{N-1}f(s^jqy)-\int fd\mu\right |<\epsilon\right\},
    \end{equation*}
    has positive $\haar_{Q_{s}}$ measure.
\end{proposition}

\begin{proof}
   Choose $x_0$ satisfying the following properties:
    \begin{enumerate}[label=(\roman*)]
        \item For $\haar_{N_s}$-a.e. $u$ the Birkhoff ergodic theorem holds for $ux_0$ with respect to $\mu$ and $s$. This is generic for $\mu$ by $N_s$-invariant.
        \item  For $Q_s$-a.e. $q$, $qx_0$ has both local stable and unstable manifolds $\mathcal W^s_{\operatorname{loc}}(qx_0;s)$ and $\mathcal W^{cu}_{\operatorname{loc}}(qx_0;s)$, respectively, and the latter manifold equals $Q_s\cdot qx_0$. This is generic for $\mu$ as this measure is $N_s$-invariant, $Q_s=\langle C_s,N_s\rangle$, and because of remark \ref{especific center unstable}.  
    \end{enumerate}

    Take $0<\eta_1<\eta$ small enough so that for any $g\in G^{<\eta_1}$ we have that $|f(gx)-f(x)|<\epsilon$ for all $x\in M$, which can be done as $f$ is uniformly continuous on $M$. 

    Let $\mathcal R$ be the nonuniformly partially hyperbolic set given by Lemma \ref{Ca closedness of R}, and let $\mathcal R=\bigcup_{l\in\N}\mathcal R^l$ be an exhaustion into regular sets as in Section \ref{Smooth ergodic theory of nonuniformly partially hyperbolic diffeomorphisms}. Since $\mathcal R$ has full $\mu$ measure, we may assume, without loss of generality, that $x_0, qx_0\in\mathcal R$ for $\haar_{Q_s}$-a.e. $q$. 

    There exists some $L_0\in\N$ such that the set
    \begin{equation*}
       Q_1:= \{q\in Q_s^{<\eta_1}:qx_0\in\mathcal R^{L_0}\}
    \end{equation*}
    has positive $\haar_{Q_s}$ measure. 

    Since for $\haar_{Q_s}$-a.e. $q$ we have that $Q_s\cdot qx_0$ is tangent to $E^s(qx_0)$ and $E^0(qx_0)$(for the distributions of $\mathcal R$) at $qx_0$, and this manifold is transversal to the stable manifold $\locstable(qx,s)$, for $\haar_{Q_s}$-a.e. $q\in Q_s^{<\eta_1}$, and hence since $\locstable(-)$ is uniformly continuous on $\mathcal R^{L_0}$, it follows that $Q_s^{<\eta_1}\cdot x_0$ is a local transversal to the family $\mathcal L^{L_0}(x)$ defined in Section \ref{Absolute continuity along stable manifolds section}. Furthermore, it follows that there exists some $\delta>0$ such that $Q_s^{<\eta_1}\cdot y$ is a local transversal to the family $\mathcal L^{L_0}(x)$, for all $y\in M$ such that $d(x,y)<0$, using again the uniform continuity of $\locstable$ on $\mathcal R^{L_0}$.

    Hence, for such $y$, the holonomy map along the family of local stable manifolds $\mathcal L^{L_0}(x)$,
    \begin{equation*}
        \psi:Q_1\cdot x_0\rightarrow Q_s^{<\eta_1}\cdot y
    \end{equation*}
    is an injection, where defined, and both itself and its inverse are absolutely continuous by theorem \ref{absolute continuity along stable}. We may choose $\delta$ small enough, so that $\psi$ above is defined in a subset of $Q_1$ of positive $\haar_{Q_s}$ measure.

    Thus if $Q_2\subseteq Q_s^{<\eta_1}$ is such that $\psi[Q_1\cdot x_0]=Q_2\cdot y$, then $Q_2$ has positive $\haar_{Q_s}$ measure. Furthermore, for all $q\in Q_2$, we have that $qy\in\mathcal W_{\operatorname{loc}}^s(\psi^{-1}(qy);s)$, and if $q_1$ is the element of $Q_1$ such that $\psi(q_1x_0)=qy$, then
    \begin{equation}\label{hjhmnbmk}
        \lim_{N\rightarrow\infty}\frac1N\sum_{j=0}^{N-1}f(s^jqy)=\lim_{N\rightarrow\infty}\frac1N\sum_{j=0}^{N-1}f(s^jq _1x_0),
    \end{equation}
    by definition of stable manifold.

    We may assume $\eta_1$ was chosen small enough so that $Q^{<\eta_1}$ is in the image of a neighborhood of $(e,e)$ in $C_s\times N_s$ on which the first map of lemma \ref{multiplication maps} restricted to this neighborhood is a diffeomorphism onto its image. Since $Q_1\subseteq Q^{<\eta_1}$ has positive $\haar_{Q_s}$ measure we may assume $q_1=c_1n_1$, where $n_1$ satisfies condition (i) with respect to $x_0$.

    \begin{equation*}
        \limsup_{N\rightarrow\infty}\frac1N\sum_{j=0}^{N-1}\left |f(s^jq_1x_0)-f(s^jn_1x_0)\right |<\epsilon,
    \end{equation*}
    as $s^jq_1=c_1s^jn_1$, $c_1\in C_s^{<\eta_1}$, and because of our choice of $\eta_1$ above with respect to $f$ and $\epsilon$. This  inequality implies 
    \begin{equation*}
        \left |\limsup_{N\rightarrow\infty}\frac1N\sum_{j=0}^{N-1}f(s^jq_1x_0)-\int fd\mu\right |<\epsilon
    \end{equation*}
    for $\haar_{Q_s}$-a.e. $q_1\in Q_1$. Hence, from (\ref{hjhmnbmk}), we obtain the result, since $Q_1\subseteq Q_s^{<\eta_1}\subseteq Q_s^{<\eta}$.
\end{proof}

We also have the following slightly more general version of \ref{local rigidty of Pinvariant measures}. Let $\alpha$ be a root of $(G,A)$ with $\alpha(s)=0$.

\begin{proposition}\label{translated local rigidty of Pinvariant measures}
     Take $\mu$ as in Proposition \ref{local rigidty of Pinvariant measures}. For $\mu$-a.e. $x$ we have that for any $u\in G^{\alpha}$,  $f\in C^0(M)$, $\eta>0$ and $\epsilon>0$, there exists  $\delta=\delta(x,u,f,\eta,\epsilon)>0$ such that for any  $y$ with $d(ux,y)<\delta$, the set
    \begin{equation*}
        \left\{q\in Q_{s}^{<\eta}: \left |\limsup_{N\rightarrow\infty}\frac1N\sum_{j=0}^{N-1}f(s^jquy)-\int fd\mu\right |<\epsilon\right\},
    \end{equation*}
    has positive $\haar_{Q_{s}}$ measure.
\end{proposition}
\begin{proof}
    Suppose $x_0$ is such that Proposition \ref{local rigidty of Pinvariant measures} holds for $ux_0$, for $\haar_{G^\alpha}$-a.e. $u$, which can be done as $\mu$ is $G^\alpha$-invariant. Pick such a $u_1$, with $d_G(u,u_1)<\eta/3$.   

    If $y\in M$ is close enough to $ux$, then $u_1u^{-1}y$ is close to $u_1x$, and thus for such $y$ we have that since $u_1u^{-1}\in Q_s$ and $d_G(u,u_1)<\eta/3$, it follows that the set of $q\in Q_s^{<\eta/3}$ satisfying the inequality of this proposition for the point $u_1y$ has positive $\haar_{Q_s}$ measure.  Therefore, the set of $q\in Q_s^{<\eta}$ satisfying the inequality of this proposition for the point $uy$ has positive $\haar_{Q_s}$ measure, and thus, we obtain the result.
\end{proof}

%we get $u\cdot  Q_1\subseteq u\cdot Q^{<\delta'}_s$, and all local stable submanifolds in the collection $\{\mathcal W_{\operatorname{loc}}^s(uqx_0,s)\}_{q\in Q_1}$ intersect $Q_s^{<\delta}\cdot y$

\subsection{Proof of Topological Measure Rigidity}

We now prove Proposition \ref{topological measure rigidity}.  

\begin{topological measure rigidity}
    Suppose $x$ is $\mu$ generic enough for Proposition \ref{translated local rigidty of Pinvariant measures} to hold with respect to $s$ and $\mu$.

    By hypothesis there exists $u_0\in (G^{-\alpha})\setminus\{e\}$ with $u_0x_0\in\supp(\mu)$. By lemma \ref{deodhar}, there are $u_1\in G^{-\alpha}$, $v_1\in G^\alpha$, and $c_0\in C_G(A)$ such that $w_\alpha=u_1v_1u_0c_0$. Since $\mu$ is invariant under $c_0$, let us assume $x_0$ is such that Birkhoff's ergodic Theorem holds for $x_0$ and $\mu$ with respect to $s$, which implies the same is true for $c_0x_0$, as $\mu^P$ is $C_G(A)$ invariant; hence, we may assume, without loss of generality, that $c_0=e_G$.

    Let $f\in C^0(M)$ and $\epsilon>0$ be arbitrary.  Choose $\delta_1>0$ such that $|f(gx)-f(x)|<\epsilon$, whenever $g\in G^{<\delta_1}$, for all $x\in M$.

    Let us assume $x_0$ is generic enough for Proposition \ref{local rigidty of Pinvariant measures} to hold for this point with respect to $s$, $\mu$ and $\alpha$; so that this proposition holds for $w_\alpha x_0$ with respect to $s$, $(w_\alpha)_\ast\mu$ and $-\alpha$.

    Take $\delta_2=\delta_2(w_\alpha x,u_1^{-1},f,\eta,\epsilon)$, as in Proposition \ref{local rigidty of Pinvariant measures}, we may assume that $\delta_2\leq\delta_1$. Choose $y$ close enough to $u_0x_0$ such that $d(v_1y, u_1^{-1}w_\alpha x_0)<\delta_2$, which can be done as $u_1^{-1}w_\alpha=v_1u_0$; and such that for some $v_2\in G^\alpha$ with $d(v_2y, u_1^{-1}w_\alpha x_0)<\delta_2$ we have that
    \begin{equation}\label{bñlfg}
        \lim_{N\rightarrow\infty}\frac1N\sum_{j=0}^{N-1}f(s^juv_2y)=\int fd\mu,
    \end{equation}
    for $\haar_{N_s}$-a.e. $u$, this holds as $u_0x_0\in\supp(\mu)$, and as $\mu$ is both $G^\alpha$ and $N_s$ invariant. Since $d(v_2y, u_1^{-1}w_\alpha x_0)<\delta_2$, it follows from Proposition \ref{translated local rigidty of Pinvariant measures} that the set of all $q\in Q_s^{<\delta_2}$ such that
    \begin{equation}\label{rmgkglbm}
        \left|\limsup_{N\rightarrow\infty}\frac1N\sum_{j=0}^{N-1}f(s^jqv_2y)-\int fd(w_\alpha)_\ast\mu\right|<\epsilon
    \end{equation}
    has positive $\haar_{Q_s}$ measure. However, if $\delta_2$ is small enough, $\haar_{Q_s}$-a.e. $q\in Q_s^{<\delta_2}$ factor uniquely as $cu$, for some $c\in C_s^{<\delta_2}$ and $u\in N_s$. Hence, there exist $c_2\in C_s^{<\delta_2}$ and $u_2\in N_s$ such that $u_2$ satisfies (\ref{bñlfg}) and $q_2:=c_2u_2$ satisfies (\ref{rmgkglbm}). By our choice of $\delta_2$ we have
    \begin{equation*}
        \limsup_{N\rightarrow\infty}\frac1N\sum_{j=0}^{N-1}\left |f(s^jq_2v_2x_0)-f(s^ju_2v_2x_0)\right |<\epsilon;
    \end{equation*}
    combining this with (\ref{bñlfg}) and (\ref{rmgkglbm}) gives 
    \begin{equation*}
        \left |\int fd\mu-\int fd(w_\alpha)_\ast\mu\right |<2\epsilon.
    \end{equation*}
    Since $f$ and $\epsilon$ were arbitrary, the result follows.
\end{topological measure rigidity}

\section{Lyapunov Exponent Transverse to Orbits and Resonance}\label{Lyapunov Exponent Transverse to Orbits and Resonance}
%For each $\alpha\in\Pi_P$ let $\phi^\alpha_t$ be the one parameter subgroup of $G$ equal to the intersection $$\bigcap_{\beta\in\Pi_P\setminus\{\alpha\}}\ker(\beta),$$
%and set $F:=\{\phi^\alpha_t:\alpha\in\Pi_P\}$.

Fix a higher rank, connected, simple, split Lie group $G$, along with a maximal torus $A\leq G$ and a minimal parabolic subgroup $P\leq G$ such that $A\leq P$. Recall we denote by $\Delta_P$ the set of positive roots with respect to $P$, and its set of simple positive roots by $\Pi_P$.

In this section we begin the proof of theorem \ref{main}. We will prove theorem \ref{main} by making the following assumption from this section on, as the first item of Theorem \ref{main} follows from the Theorem of Stuck and Zimmer, see Theorem \ref{sz}.

\begin{assumption}\label{main assumption}
    Suppose there exists a $C^3$ closed manifold $M$ of dimension $\dim(G)+1$, along with a $C^3$ locally free action $G\acts M$ having a $P$-invariant probability measure $\mu^P$ on $M$, such that the following conditions hold
    \begin{enumerate}
        \item for every $\alpha\in\Pi_P$ there exists some $s\in\ker(\alpha)\setminus\{e\}$ such that $\mu^P$ is $s$-ergodic, i.e., $\ker(\alpha)\acts (M,\mu^P)$ is ergodic for all $\alpha\in\Pi_P$, because of Theorem \ref{pughshub};
        \item $M$ has no invariant probability measure under the action $G\acts M$.
    \end{enumerate}
\end{assumption}

The assumption on ergodicity by singular elements of $A$ appears in the work of Katok-Spatzier \cite[Section~5]{katokspatzier}; and the work of Nevo and Zimmer, see \cite[Theorem~A]{nzinventiones} \cite[Theorem~3]{nzann}.

Let $Q_0$ be the stabilizer of $\mu^P$ in $G$, i.e., the set $\{g\in G:g_\ast\mu^P=\mu^P\}$. Then $Q_0$ is a proper, because of our assumption, Assumption \ref{main assumption}, parabolic subgroup of $G$ as $P\leq Q_0$. Denote by $\Delta_{Q_0}$ the set of all nonzero roots $\beta\in A^\ast$ such that the associated unipotent group $G^\beta$ is contained in $Q_0$ in $G$; and denote its compliment in the set of all roots of $(G,A)$ by $\Delta_{Q_0}^c$.

\begin{remark}\label{A-ergodicity of muP}
Since in particular $\mu^P$ is $s$-ergodic by an element of $A\leq P$ ,  it follows that $\mu^P$ is both $A$ and $P$-ergodic. Since we are assuming $G\acts M$ has no $G$-invariant probability measure, it follows from proposition \ref{uniquePergodicity} that $\mu^P$ has to be supported on a minimal subset $\m$ of $M$, and that it is the unique $P$-invariant measure supported on this minimal set. 
\end{remark}

We have the following immediate consequence of the Theorem of Pugh and Shub, Theorem \ref{pughshub}:

\begin{corollary}\label{corollarypughshub}
    All elements off a countable family of hyperplanes in $A$ act $\mu^P$-ergodically on $M$.
\end{corollary}

\begin{definition}
As $\mu^P$ is $A$-ergodic, there is a corresponding decomposition of $T_xM$ into Osedelets' distributions $\bigoplus_{j=1}^mE^{\lambda_j}(x)$, for $\mu$-a.e. $x$,  where $\lambda_j\in A^*$ for all $j=1,\ldots, m$. Because of our codimension one hypothesis, there exists exactly one Lyapunov exponent in this decomposition whose corresponding Osedelets distribution is not tangent to the $G$-orbits, which we denote by $\lo$.
\end{definition}

\begin{remark}\label{unique ergodicity for other minimal parabolics}
    We have fixed a choice of a minimal parabolic subgroup of $G$; however, for any other minimal parabolic $P'$ in $G$ with respect to $A$, we denote by $\mu^{P'}$ the unique $P'$-invariant probability measure on $\m$, where uniqueness follows from our assumption \ref{main assumption} and proposition \ref{uniquePergodicity}; so that if $\mu^P$ is also $P'$-invariant, then $\mu^P=\mu^{P'}$.
\end{remark}

In this section we first prove that for any $a\in A$ such that its unstable subgroup, $N_a$, is contained in $Q_0$, the Lyapunov exponent $\lambda^T_{\mu^P}(a)$ of the $A$-ergodic probability measure $\mu^P$ satisfies the inequality $\lambda^T_{\mu^P}(a)<\beta(a)$ for all $\beta\in\Delta$. It will then follow that the Lyapunov exponent $\lambda^T_{\mu^P}(a)$ will be the lowest Lyapunov exponent of the $a$-ergodic probability measure $\mu^P$. This will imply that the Osedelec's distributions with respect to the transverse Lyapunov exponent $\lambda^T_{\mu^P}(a)$ will integrate to a one dimensional Pesin manifolds, transverse to the $G$-orbits of the given action, which we shall call \emph{transverse Pesin manifolds with respect to} $a$, and denote by $\mathcal W^T(x;a)$, for $\mu^P$-a.e. $x$ for which such Pesin manifold exists. We note that even though $\mu^P$ is supported on the minimal set $\m$, the transverse Pesin manifolds will not be contained in this set, unless the action $G\acts M$ is minimal.

\subsubsection{Invariance of ergodic components of $\mu^P$}

Let $a\in\mathcal W^P$ be nontrivial, and let $\mu^P=\int\nu_xd\mu^P(x)$ be the $a$-ergodic decomposition of $\mu^P$. Recall we denote by $N_a$ the unstable subgroup of $a$ in $G$, see (\ref{Na and La}). The following fact uses the Hopf argument, and is inspired by \cite{dh}[Proposition~8.4]
\begin{lemma}\label{invariance of ergodic components}
    For $\mu^P$-a.e. $x$ we have that $\nu_x$ is $N_s$-invariant.
\end{lemma}
\begin{proof}
    %Notice that since $\lo(a)<0$, by lemma \ref{negativity of transversal lyapunov exponent on positive weyl chambers}, we get that for $\mu^P$-a.e. $x$ the unstable manifold $\mathcal W^u(x;a)$ equals the orbit $N_a\cdot x$.
    
    For any measurable partition $\xi$ of $M$ subordinate to the $P$-orbits the conditional measure of $\mu^P$ along $\xi(x)$ is equivalent to the Riemannian volume on the $H$-orbit $H\cdot x$, by \ref{invariance for subgroups with orbits in unstable}, for $\mu^P$-a.e. $x$. Because of the Hopf argument, see \cite[Proposition~8.12]{brown2019entropylyapunovexponentsrigidity}, the partition into unstable manifolds with respect to $a$ refines the ergodic decomposition of $\mu^P$. Thus, for $\mu^P$-a.e. $x$, the conditional measure of $\nu_x$ along $\xi(y)$ is equal to the conditional measure of $\mu^P$ for this element, for $\nu_x$-a.e. $y$. Hence, applying \ref{invariance for subgroups with orbits in unstable}, it follows that $\nu_x$ is $N_a$-invariant for $\mu^P$-a.e. $x$.
\end{proof}

\subsection{Negativity of transversal Lyapunov exponent in $\intW$} 

\subsubsection{Minimal resonant codimension}\label{Minimal resonant codimension}
We make use of the notion of minimal resonant codimension of $G$ defined in \cite{brhw}. Unlike in \cite{brhw}, we do not need to consider coarse roots of $G$ since we are assuming $G$ is split. 

\begin{definition}
    The \emph{minimal resonant codimension} of $G$, which we denote by $r(G)$, is the minimal dimension of the manifold $G/Q$, for all proper parabolic subgroups $Q\leq G$.
\end{definition}

Since $G$ is higher rank, we have that $r(G)\geq 2$, see \cite[Example~1.3]{brhw}.

\subsubsection{Kernel of transverse Lyapunov exponent does not intersect $\mathcal W^P$}

The following lemma  is inspired by ideas appearing in \cite{bfh} and \cite{brhw} to construct $G$-invariant measures on suspension spaces of actions of lattices on manifolds. This lemma uses the higher rank hypothesis on $G$.

\begin{lemma}\label{kernel of transverse exponent touches not weyl chamber}
    $\ker(\lo)\cap\intW=\{0\}$.
\end{lemma}
\begin{proof}
    Suppose, to get a contradiction, that the intersection $\ker(\lo)\cap\intW$ is not trivial. Pick $a\in\intW$ such that $\mu^P$ is $a$-ergodic, which can be done by Corollary \ref{corollarypughshub}. Let us show $\mu^P$ is $G^{-\beta}$ invariant for all $\beta\in\Delta_P$. This will imply that the stabilizer of $\mu^P$, $Q_0$, has resonant codimension at most $1$, which contradicts the fact that $r(G)\geq 2$.

    We have 
    \begin{equation*}
        \sum_{\beta\in\Delta_P}\beta(a)\geq h_{\mu^P}(a^{-1})=h_{\mu^P}(a)=\sum_{\beta\in\Delta_P^+}\beta(a),
    \end{equation*}
    where the first inequality follows from the Margulis-Ruelle inequality, see Theorem \ref{Ruelle-Margulis inequality}, and the last equality follows from lemma \ref{invariance from maximal entropy} and $P$-invariance of $\mu^P$; hence, $h_{\mu^P}(a^{-1})=\sum_{\beta\in\Delta_P}\beta(a)$, and therefore $\mu^P$ is invariant under $G^{-\beta}$ for all $\beta\in\Delta_P$, so $\mu^P$ is $G$-invariant, which contradicts Assumption \ref{main assumption}.

    Using that $r(G)\geq 2$ and the previous argument, we get that $\ker(\lo)$ does not intersect the interiors of the codimension one faces of $\overline{\mathcal W^P}$.

    \begin{lemma}\label{kernel of transversal no faces}
        For each $\alpha\in\Pi_P$ we have $\ker(\lo)\cap\operatorname{Int}(\ker(\alpha)\cap\overline{\mathcal W^P})=\emptyset$.
    \end{lemma}

    \begin{proof}
    Otherwise, if there is an element $s$ in this intersection, let $\nu_x$ be a generic $s$-ergodic component of $\mu^P$, so that $\nu_x$ is $G^\beta$ invariant for all $\beta\in\Delta_P\setminus\{\alpha\}$, by Lemma \ref{invariance of ergodic components}, as $\beta(s)>0$; an entropy argument as the one in the proof of the previous lemma implies $\nu_x$ is $G^{-\beta}$ invariant for all $\beta\in \Delta_P\setminus\{\alpha\}$. As $r(G)\geq 2$, this implies $\nu_x$ is $G$-invariant, and hence so is $\mu^P$, contradicting Assumption \ref{main assumption}.
    \end{proof}

%If there existed $\alpha'\neq\alpha$ satisfying the properties of our lemma, so in particular $-\alpha' \in \Pi_P^+$ and $\alpha'\notin\Delta_{Q_0}$, we would get that $\mu_P$ would be $G^{\alpha+\alpha'}$-invariant because of the high entropy method, and so by construction of $Q_0$ we would get that $\alpha + \alpha' \in \Delta_Q$, but since $-\alpha' \in \Pi_P^+\subseteq\Delta_{Q_0}$, this would imply $\mu_P$ to be $G^{\alpha}$-invariant, since $G^{\alpha}=[G^{\alpha+\alpha'},G^{-\alpha'}]$, i.e., $\alpha\in\Delta_{Q_0}$, a contradiction.
\end{proof}

Recall we denote by $w_\alpha$ the element in the Weyl group $W_{(G,A)}$ corresponding to the reflection of $A$ about the hyperplane $\ker(\alpha)$. For the proof of proposition \ref{entropic measure rigidity} we use our ergodic assumption for singular elements in \ref{main assumption}.

\subsubsection{Non-multiple roots and minimal parabolics for split groups}
We drop our fixed conventions in this subsubsection to prove the following fact about semisimple Lie groups.

Let $G$ be a semisimple Lie group; $A\leq G$ be a maximal torus. Let $Q$ be a parabolic subgroup of $G$ containing $A$. Let us denote by $\Delta_Q$ the set of roots of $(G,A)$ appearing in the root decomposition of $Q$. We have the following observation following from the  Levi decomposition of such parabolics, see \cite[Theorem~B.2]{knappbeyond}. 

\begin{lemma}\label{minimal parabolics containing roots}
    For any $\alpha\in\Delta_Q$, there exists a minimal parabolic $P$ such that $\alpha\in\Delta_P$.
\end{lemma}
\begin{proof}
If $G^\alpha$ is contained in the unipotent radical of $Q$, then in this case $\alpha\in\Delta_P$ for any minimal parabolic $P\leq Q$.

If $G^\alpha$ is contained in the semisimple part of $Q$, there is a minimal parabolic of the corresponding simple factor of this semisimple part containing $G^\alpha$, which yields a minimal parabolic of $Q$ with the required property.
\end{proof}

\subsection{Maximal negativity of transversal Lyapunov exponent}

The following proposition is one of the main tools of the whole paper, it is the main ingredient for the construction of the embedded submanifold in section \ref{Proof of theorem when DeltaminusDeltaQgeq3}. The proof of this proposition, under our assumption \ref{main assumption}, uses the higher rank and split hypothesis on $G$. It also uses the assumption that $\mu^P$ is ergodic for $\ker(\alpha)$ for all $\alpha\in\Pi_P$.

\begin{proposition}\label{transverse lyapunov exponent strictly maximally negative}
    For any  $a\in\intW$ we have
    \begin{equation}
        \beta(a)<\lo(a^{-1})
    \end{equation}
    for all $\beta\in\Delta_P$
\end{proposition}

Recall $r(G)$ denotes the resonant codimension of $G$. When $r(G)\geq 3$ we also have the following property which will be quite instrumental in section \ref{Proof of theorem when DeltaminusDeltaQgeq3} to construct the smooth factor map. Recall $Q_0=\operatorname{Stab}(\mu^P)$. 

\begin{proposition}\label{transverse lyapunov exponent strictly maximally negative for least singular elements}
    Suppose $|\Delta\setminus\Delta_{Q_0}|\geq 3$. Let  $\alpha\in\Pi_P$. Then for any $s\in\operatorname{Int}(\ker(\alpha)\cap\overline{\mathcal W^P})$ we have
    \begin{equation}
        \beta(s)<\lo(s^{-1})
    \end{equation}
    for all $\beta\in\Delta_P$.
\end{proposition}

We will give a proof of proposition \ref{transverse lyapunov exponent strictly maximally negative} only when $r(G)=2$. In the case $r(G)\geq 3$, proposition \ref{transverse lyapunov exponent strictly maximally negative} follows immediately from proposition \ref{transverse lyapunov exponent strictly maximally negative for least singular elements}. Indeed, in this case $|\Delta\setminus\Delta_{Q}|\geq 3$ holds for any parabolic $Q\leq G$ by definition of $r(G)$; and $\intW$ is contained in the convex hull of the codimension one faces of $\overline{\mathcal W^P}$.

\begin{remark}\label{reduction to inequality on kernel}
    The constructions and all the properties listed in the entirety of the paper will also hold in the case $r(G)=2$ if the given action satisfies proposition \ref{transverse lyapunov exponent strictly maximally negative for least singular elements}, as every application of the assumption $|\Delta\setminus\Delta_{Q_0}|\geq 3$ in the whole paper is done via this proposition.
\end{remark}

\begin{example}
    We need $G$ to be higher rank for the propositions above to hold. Indeed, if $G=\SL(2,\R)$, and $\Gamma$ is a torsion free uniform lattice of $G$, and $P$ is the subgroup of upper triangular matrices, consider the canonical diagonal action $G\acts M:=G/\Gamma\times G/P$, which is smooth, locally free, and with codimension one orbits. Then, $P\acts M$ is uniquely ergodic, with corresponding invariant probability measure, $\mu:=\haar_{G/\Gamma}\times\delta_{eP}$, where $eP$ is the identity coset of $G/P$. This measure is $\ker(\alpha)$-ergodic for all roots $\alpha$, by the Howe-Moore Theorem, but its transverse Lyapunov is the zero functional.
\end{example}

\subsubsection{Negativity of transverse Lyapunov exponent}

Let us make the following observation, under the assumption that $\lo$ is resonant with some positive root with respect to $P$. Given any root $\beta\in\Delta$ we denote by  $\mathcal W^\beta$ the partition of $M$ given by the $G^\beta$ orbits. 

\begin{lemma}\label{entropy along unique coarse lyapunov exponent}
    If $\lo$ is resonant to some $\beta\in\Delta_P$, then the coarse Lyapunov exponent, denoted by $\chi$, of $\lo$ equals $\{\beta,\lo\}$ and 
    $\hmup(a|\mathcal W^\chi)\geq\beta(a)$ for all $a\in\intW$.
    \end{lemma}

\begin{proof}
    The first statement follows from the fact that $G$ acts locally freely with leaves of codimension one combined with the assumption that $G$ is split. 

    For the second statement, it is enough to prove it for all $a\in\intW$ such that $\mu^P$ is $a$-ergodic because of theorem \ref{pughshub}, and since $\hmup(-|\mathcal W^P)$ is linear on $\intW$, as no coarse Lyapunov exponent changes sign in this open cone by the hypothesis, because of theorem \ref{ledrappieryoung2}.

    We have $\lo=c\beta$ for some $c>0$.

    If $c\leq 1$, then from Theorem \ref{ledrappieryoung2} we have that $\hmup(a|\mathcal W^\chi)\geq\gamma(\mu^P|W^\beta)\beta(a)$, where $\gamma(\mu^P|W^\beta)$ is the Hausdorff dimension of $\mu^P$ with respect to the partition by $G^\beta$ orbits. If $0<c<1$, then from Theorem \ref{ledrappieryoung2}, we get the equality immediately, as $\mu^P$ is $G^\beta$ invariant.

    If $c=1$, the inequality follows from Theorem \ref{ledrappieryoung2}, as $\gamma(\mu^P|\mathcal W^\chi)\geq 1$, since $\mathcal W^\chi$ it is refined by the partition into $G^\beta$-orbits. 

    If $c>1$, then applying theorem \ref{ledrappieryoung2} again, 
    \begin{equation*}
    \begin{split}
        \hmup(a|\mathcal W^\chi)&=\lo(a)\gamma_2(\mu^P|\mathcal W^\chi)+\beta(a)\gamma_1(\mu^P|\mathcal W^\chi)\\
        &>\beta(a)[\gamma_2(\mu^P|\mathcal W^\chi)+\gamma_1(\mu^P|\mathcal W^\chi)]
    \end{split}
    \end{equation*}
    and $\gamma_2(\mu^P|\mathcal W^\chi)+\gamma_1(\mu^P|\mathcal W^\chi)$ equals the Hausdorff dimension of $\mu^P$ along $\mathcal W^\chi$, which is at least $1$, hence we obtain the lemma in this case as well.
\end{proof}

In the next lemma we prove that the transversal Lyapunov exponent is negative in the interior of the Weyl chamber corresponding to $P$. 

\begin{lemma}\label{negativity of transversal lyapunov exponent on positive weyl chambers}
    For any $a\in\intW$ we have $\lambda_{\mu^P}^T(a)<0$.
\end{lemma}
\begin{proof}
    It is enough to prove this for all $a\in\intW$ for which $\mu^P$ is $a$-ergodic, since all such points are dense in $\intW$, because of theorem \ref{pughshub}. Fix such an $a$. 
    
    Suppose, to get a contradiction, that $\lo(a)\geq 0$. There are two cases, depending on whether $\lo$ is resonant to a root or not. Denote by $\chi$ the coarse Lyapunov exponent of $\lo$.

    \begin{case negativity of transversal lyapunov exponent on positive weyl chambers} $\lo$ is not resonant to a root. We then have that $\chi=\{\lo\}$, and using theorem \ref{entropy product structure}
    \begin{equation}\label{entropy product no resonance case}
        \hmup(a)=\hmup(a|\mathcal W^\chi)+\sum_{\beta\in\Delta_P}\hmup(a|\mathcal W^\beta)=\hmup(a|\mathcal W^\chi)+\sum_{\beta\in\Delta_P}\beta(a),
    \end{equation}
    where the last equality follows as $\mu^P$ is $P$ invariant. Applying theorem \ref{entropy product structure} again,  we get 
    \begin{equation*}
        \hmup(a^{-1})=\sum_{\beta\in\Delta_P}\hmup(a^{-1}|\mathcal W^{-\beta}).
    \end{equation*}
    As $\hmup(a)=\hmup(a^{-1})$, using  (\ref{entropy product no resonance case}) and the previous equality, canceling accordingly, we obtain that 
    \begin{equation}\label{cancellations entropy product no resonance case}
        \hmup(a|\mathcal W^\chi)+\underset{-\beta\in\Delta_{Q_0}^c}{\sum_{\beta\in\Delta_P}}\beta(a)=\underset{-\beta\in\Delta_{Q_0}^c}{\sum_{\beta\in\Delta_P}}\hmup(a^{-1}|\mathcal W^{-\beta}),
    \end{equation}
    as for each $\beta\in\Delta_P$ such that $-\beta\in\Delta_{Q_0}$ we have that
    \begin{equation*}
        \hmup(a|\mathcal W^\beta)=\beta(a)=(-\beta)(a^{-1})=\hmup(a^{-1}|\mathcal W^{-\beta}),
    \end{equation*}
    as $\mu^P$ is invariant by both $G^{\beta}$ and $G^{-\beta}$, by construction of $Q_0$; and theorem \ref{ledrappier young conditional entropy}.  From (\ref{cancellations entropy product no resonance case}) we have that 
    \begin{equation*}
        \underset{-\beta\in\Delta_{Q_0}^c}{\sum_{\beta\in\Delta_P}}\hmup(a^{-1}|\mathcal W^{-\beta})\geq \underset{-\beta\in\Delta_{Q_0}^c}{\sum_{\beta\in\Delta_P}}\beta(a),
    \end{equation*}
    however, from theorem \ref{ledrappier young conditional entropy}, $\hmup(a^{-1}|\mathcal W^{-\beta})\leq \beta(a)$ for all $\beta\in\Delta_P$, and thus $\hmup(a^{-1}|\mathcal W^{-\beta})=\beta(a)$ for all $\beta\in\Delta_P$ such that $-\beta\notin\Delta_{Q_0}$, which implies by theorem \ref{ledrappier young conditional entropy} that $\mu^P$ is $G^{-\beta}$ invariant for all such $\beta$'s, contradicting the definition of $Q_0$.
    \end{case negativity of transversal lyapunov exponent on positive weyl chambers}

    \begin{case negativity of transversal lyapunov exponent on positive weyl chambers} $\lo$ is resonant to a root $\beta_0$. In this case we have $\chi=\{\lo,\beta_0\}$, because of the codimension one hypothesis on the orbits of $G\acts M$, and as $G$ is a split group. Note that $\beta_0\in\Delta_P$, as we are assuming $\lo$ is nonnegative in a point of the interior of the Weyl chamber $\mathcal W^P$. As $\lo(a)\geq 0$, from theorem \ref{entropy product structure} we get
    \begin{equation*}
        \hmup(a)=\hmup(a|\mathcal W^\chi)+\underset{\beta\neq\beta_0}{\sum_{\beta\in\Delta_P}}\hmup(a|\mathcal W^\beta)=\hmup(a|\mathcal W^\chi)+\underset{\beta\neq\beta_0}{\sum_{\beta\in\Delta_P}}\beta(a)
    \end{equation*}
    and 
    \begin{equation*}
        \hmup(a^{-1})=\sum_{\beta\in\Delta_P}\hmup(a^{-1}|\mathcal W^{-\beta});
    \end{equation*}
    and thus using that $\hmup(a)=\hmup(a^{-1})$ we obtain,  
    arguing similarly as in the previous case, that
    \begin{equation}\label{cancellations entropy product resonance case}
        \hmup(a|\mathcal W^\chi)+\underset{\beta\neq\beta_0,-\beta\in\Delta_{Q_0}^c}{\sum_{\beta\in\Delta_P}}\beta(a)=\underset{\beta\neq\beta_0,-\beta\in\Delta_{Q_0}^c}{\sum_{\beta\in\Delta_P}}\hmup(a^{-1}|\mathcal W^{-\beta})+\hmup(a^{-1}|\mathcal W^{-\beta_0}).
    \end{equation}
    From lemma \ref{entropy along unique coarse lyapunov exponent} we have $\hmup(a|\mathcal W^\chi)\geq\beta_0(a)$; we have $\beta_0(a)\geq \hmup(a^{-1}|\mathcal W^{-\beta_0})$ as well, hence, from the previous equality we get 
    \begin{equation}\label{sum of entropies resonance case}
       \underset{\beta\neq\beta_0,-\beta\in\Delta_{Q_0}^c}{\sum_{\beta\in\Delta_P}}\beta(a)\leq \underset{
       \beta\neq\beta_0,-\beta\in\Delta_{Q_0}^c}{\sum_{\beta\in\Delta_P}}\hmup(a^{-1}|\mathcal W^{-\beta})   
    \end{equation}
    Both sums will be non-vacuous if we prove that $\{\beta\in\Delta_P:\beta\neq\beta_0, -\beta\in\Delta^c_{Q_0}\}$ is not empty. Suppose it is. Then we would have $-\beta\in\Delta_{Q_0}$ for all $\beta\in\Delta_P\setminus\{\beta_0\}$, and thus $\Delta\setminus\{\beta_0\}\subseteq\Delta_{Q_0}$. However, as the resonant codimension of $G$ is at least $2$, see section \ref{Minimal resonant codimension}, we would get that $Q_0=G$, and thus $\mu^P$ would be $G$-invariant, which contradicts assumption \ref{main assumption}. 
    
    However, because of theorem \ref{ledrappier young conditional entropy}, we get from (\ref{sum of entropies resonance case}) that $\hmup(a^{-1}|\mathcal W^{-\beta})=\beta(a)$ for all $\beta$ indexed by the sums in (\ref{sum of entropies resonance case}), and since this index set is not empty, using theorem \ref{ledrappier young conditional entropy}, $\mu^P$ is $G^{-\beta}$ invariant for some $\beta\in\Delta_P$ such that $-\beta\notin\Delta_{Q_0}$, contradicting the construction of $Q_0$.

    Therefore, we must have $\lo(a)<0$ in this case as well.
    \end{case negativity of transversal lyapunov exponent on positive weyl chambers}
\end{proof}

\begin{remark}\label{resonant root can't be in root complement of Q0}
Note that if $\lo$ is resonant with some root, by the previous lemma, it must belong to $\Delta_{Q_0}^c$, for if it were resonant to some $\beta_0\in\Delta_{Q_0}$, there would exist some minimal parabolic $P'$ of $G$ contained in $Q_0$ such that $G^{\beta_0}\leq P'$, by using the Levi decomposition of $Q_0$, so that $\mu^P=\mu^{P'}$ by remark  \ref{unique ergodicity for other minimal parabolics}, and by considering any $a\in\operatorname{Int}(\mathcal W^{P'})$, we would get a contradiction using the previous lemma.     
\end{remark}

%If $\beta_0$ is such a root, we must have that the corresponding coarse Lyapunov exponent $\chi$ is equal to $\{\lo, \beta_0\}$, as $G$ is a split group and as there is only one Lyapunov exponent of the $A$-ergodic[MAUTNER PHENOMENOM] probability measure $\mu^P$, that is not a root.

\subsubsection{Maximal negativity  of transversal Lyapunov exponent}

For the minimal parabolic $P\leq G$, let $\Pi_P$ be the set of simple roots of the set of positive roots associated to $P$, $\Delta_P$. Denote by $\beta_P^H$ the \emph{highest root} with respect to $P$, given by the sum of all elements of $\Delta_P^+$, then $\beta_P^H$ is greater than or equal to all elements in $\Delta_P^+$, i.e., for any $\beta\in\Delta_P^+$, $\beta_P^H-\beta$ is the sum of elements in $\Delta_P$, with coefficients equal to either 0 or 1, as $G$ is a split simple Lie group.

\begin{remark}\label{measure can't be invariant under negation of highest root}
    We must have that $-\beta^H_P\in\Delta_{Q_0}^c$. For otherwise, $\mu^P$ would be invariant under $G^{-\beta^H_P}$, but this would imply $\mu^P$ would be invariant under $G^{-\beta}$ for all $\beta\in\Delta_P$, as any such $\beta$ is either $\beta^H_P$ itself or can be written as $\beta^H_P-\alpha$ for some $\alpha\in\Delta_P$, which implies $[G^{-\beta^H_P},G^\alpha]=G^{-\beta}$, which would yield $G^{-\beta}$ invariance as $\mu^P$ is $G^\alpha$ invariant as well. Since $\mu^P$ is already $P$-invariant, this would mean $\mu^P$ is $G$-invariant, which would contradict assumption \ref{main assumption}. 
\end{remark}

\begin{remark}\label{resonance to negative of highest root implies stabilizer is P}
Suppose $\lo$ is resonant to $-\beta^H_P$. Then for this case we must have that $\operatorname{Stab}(\mu^{P})=P$, i.e., $Q_0=P$. For otherwise, $\mu^{P}$ would equal $\mu^{P_1}$ for some minimal parabolic $P_1$ with respect to $A$ different to $P$, because of lemma \ref{minimal parabolics containing roots}, and then $\mu^{P}$ would be invariant by the element of the Weyl group sending the Weyl chamber $\mathcal W^{P}$ to $\mathcal W^{P_2}$, but such an element must send $\beta^H_{P}$ to $\beta^H_{P_2}$, which are different roots, a contradiction.    
\end{remark}
 
Recall that since $G$ is a split simple Lie group and $\dim(M)=\dim(G)+1$, if $\lo$ is resonant to a root $\beta_0$, denote by $\chi=\{\lo,\beta_0\}$ its coarse Lyapunov exponent. In this setup we have the following fact.

\begin{lemma}\label{no full entropy along coarse Pesin manifold}
    For all $s\in\overline{\mathcal W^P}$ we have
    \begin{equation*}
        \lo(s^{-1})+\beta_0(s^{-1})>\hmup(s^{-1}|\mathcal W^\chi).
    \end{equation*}
\end{lemma}
\begin{proof}
    Recall $\lo(s)<0$ for all $s\in\mathcal W^P$ by theorem \ref{negativity of transversal lyapunov exponent on positive weyl chambers}, i.e., $\chi(s)<0$.

    By Ruelle's inequality, we have $\lo(s^{-1})+\beta_0(s^{-1})\geq \hmup(s^{-1}|\mathcal W^\chi)$. Suppose, to get a contradiction, that we have equality. 
    
    We may assume, without loss of generality, throughout the proof, that $\mu^P$ is a generic enough $s$-ergodic component of itself, which can be done as $h_{-}(s^{-1}|\mathcal W^\chi)$ is a convex function.    
    
    Let $x_0$ be a generic enough point of $\mu^P$. Since the action $G\acts M$ is at least $C^2$, it follows that the coarse Pesin manifold $\mathcal W^\chi(x_0)$ is an injectively immersed $C^2$ submanifold of $M$ of dimension $2$ such that for any $y\in M$, the injectively immersed submanifold $G^{\beta_0}\cdot y$ is contained in $\mathcal W^\chi(x_0)$. 

    There exists a diffeomorphism $\phi:\R^2\rightarrow\mathcal W^\chi(x_0)$ with the following properties:
    \begin{enumerate}
        \item $\phi(0)=x_0$;
        \item for every $y\in\R$, $\phi(y,-):\R\rightarrow U$ is a diffeomorphism onto the connected component of the intersection $G^{\beta_0}\cdot\phi(y,0)\cap U$ containing $\phi(y,0)$.
    \end{enumerate}
    Since $ \hmup(s^{-1}|\mathcal W^\chi)=\lo(s^{-1})+\beta_0(s^{-1})$, we get by theorem \ref{ledrappier young conditional entropy} that the conditional measure of $\mu^P$ along $\mathcal W^\chi$, $(\mu^P)_{x_0}^{\mathcal W^\chi}$, is absolutely continuous. As $\phi$ is a $C^1$ diffeomorphism, we can apply Fubini's theorem and obtain a disintegration of $(\mu^P)_{x_0}^{\mathcal W^\chi}$ along the partition of $\mathcal W^\chi$ by its $G^{\beta_0}$ orbits such that for $(\mu^P)_{x_0}^{\mathcal W^\chi}$-a.e. $y$ we have that $[(\mu^P)_{x_0}^{\mathcal W^\chi}]^{\mathcal W^{\beta_0}}_y$ is absolutely continuous. Since $x_0$ was chosen to be generic enough for $\mu^P$, the above implies that the decomposition of $\mu^P$ along the $G^{\beta_0}$ orbits is absolutely continuous. This implies that $\mu^P$ is $G^{\beta_0}$-invariant by proposition \ref{haarconditionals}, and therefore, since $x_0$ was arbitrarily $\mu^P$ generic, it follows that $\mu^P$ is $G^{\beta_0}$ invariant, so $\beta_0\in\Delta_{Q_0}$, which contradicts remark \ref{resonant root can't be in root complement of Q0}. 
\end{proof}

As indicated before, we prove proposition \ref{transverse lyapunov exponent strictly maximally negative} only in the case $r(G)=2$, i.e., when $G$ is locally isomorphic to $\SL(3,\R)$. 
\s

\begin{transverse lyapunov exponent strictly maximally negative when r(G)=2}
We take $P$ to be the closed subgroup of $G$ corresponding to the Lie subalgebra of $\lie(G)=\mathfrak{sl}(3,\R)$ corresponding to the upper triangular matrices of $\SL(3,\R)$. Denote by $\beta_{i,j}$, $i,j=1,2,3$, its corresponding roots, with respect to the subgroup of diagonal matrices.

The leafwise measure $(\mu^P)^{G^{\beta_{3,1}}}_x$ is atomic for $\mu^P$-a.e. $x$ by applying theorem \ref{high entropy method}, indeed, since $\mu^P$ is in particular $P$, this theorem would imply $\mu^P$ to be invariant under both $G^{\beta_{2,1}}$ and $G^{\beta_{3,2}}$, so $\operatorname{Stab}(\mu^P)$ would be a parabolic subgroup of $G$ containing both $P$ and the group $\langle G^{2,1},G^{3,2}\rangle$, which is impossible since the minimal codimension dimension of $G$, $r(G)$, is $2$.

There are two cases.

\begin{transverse lyapunov exponent strictly maximally negative when r(G)=2 cases}
    $\lo$ is resonant to a root.
\end{transverse lyapunov exponent strictly maximally negative when r(G)=2 cases}

As in this case $G$ is locally isomorphic to $\SL(3,\R)$, we have that $\beta^{3,1}$ is the only root such that $\lo$ can be resonant to, because of both lemmas \ref{kernel of transverse exponent touches not weyl chamber} and \ref{negativity of transversal lyapunov exponent on positive weyl chambers}. Thus, the coarse Lyapunov exponent of $\lo$ equals $\chi=\{\lo, \beta^{3,1}\}$. In this case $Q_0=\operatorname{Stab}(\mu^P)=P$ by remark \ref{resonance to negative of highest root implies stabilizer is P}, and thus $|\Delta\setminus\Delta_{Q_0}|=3$. Thus, in this case, the proposition follows from proposition \ref{transverse lyapunov exponent strictly maximally negative for least singular elements}

\begin{transverse lyapunov exponent strictly maximally negative when r(G)=2 cases}
$\lo$ is not resonant to a root.
\end{transverse lyapunov exponent strictly maximally negative when r(G)=2 cases}

In this case, because of  lemma \ref{kernel of transverse exponent touches not weyl chamber}, $\ker(\lo)$ intersects the interior of one of the two adjacent Weyl chambers to $\intW$. Let us assume, without loss of generality, that this adjacent Weyl chamber has $\ker(\beta_{1,2})$ as a face, the other case is analogous, and denote by $P_1$ its corresponding minimal parabolic, so that $\Delta_{P_1}=\{\beta_{2,1},\beta_{1,3},\beta_{2,3}\}$.

Choose $b\in\ker\lo$ such that $\beta_{1,3}(b)>0$. 

We claim that $\hmup(b|W^{\beta_{2,1}})=0$. Suppose not. Then $\hmup(b^{-1}|\mathcal W^{\beta_{3,2}})=0$, for otherwise we would get by theorem \ref{high entropy method} that $\mu^P$ would be $G^{\beta_{3,1}}$ invariant and  this is impossible by \ref{measure can't be invariant under negation of highest root}. Then by theorem \ref{entropy product structure},
\begin{equation*}
    \hmup(b)=\hmup(b|\mathcal W^{\beta_{2,1}})+\beta_{1,3}(b)+\beta_{2,3}(b)\text{, }\hmup(b^{-1})=\beta_{1,2}(b^{-1}),
\end{equation*}
since $\hmup(b)=\hmup(b^{-1})$ this implies
\begin{equation*}
    \hmup(b|\mathcal W^{\beta_{2,1}})+\beta_{1,3}(b)+\beta_{2,3}(b)+\beta_{1,2}(b)=0,
\end{equation*}
which is a contradiction since $\beta_{2,3}+\beta_{1,2}=\beta_{1,3}$ and $\beta_{1,3}(b)\neq 0$. Therefore $\hmup(b|\mathcal W^{\beta_{2,1}})=0$.

Thus, for any $a\in\intW$ such that $\mu^P$ is $a$-ergodic, we have by theorem \ref{entropy product structure}, since $\lo$ is not resonant to a root and by lemma \ref{negativity of transversal lyapunov exponent on positive weyl chambers}, that
\begin{equation*}
    \hmup(a^{-1})=\hmup(a^{-1}|\mathcal W^{\lo})+\hmup(a^{-1}|\mathcal W^{G^{\beta_{3,2}}}),
\end{equation*}
and as $\hmup(a^{-1})=\hmup(a)=(\beta_{1,2}+\beta_{1,3}+\beta_{2,3})(a)$ we get 
\begin{equation*}
    \hmup(a^{-1}|\mathcal W^{\lo})\geq (\beta_{1,2}+\beta_{1,3})(a)
\end{equation*}
as $\hmup(a^{-1}|\mathcal W^{\beta_{3,2}})\leq\beta_{2,3}(a)$.

Thus, as $\beta_{1,2}(a)\neq 0$, we obtain that $\lo(a^{-1})>\beta_{1,3}(a)$, and as $\beta_{1,3}$ is the highest root with respect to $P$, we are done.

\end{transverse lyapunov exponent strictly maximally negative when r(G)=2}

For the proof of Proposition \ref{transverse lyapunov exponent strictly maximally negative for least singular elements} we need the following observation. The following proposition is the first part of the proof of our main result, Theorem \ref{main}, where we use the ergodic hypothesis appearing in it. 

\begin{lemma}\label{entropy ergodic components}
     For all $\beta\in\Delta$ such that $\beta\notin\Delta_{Q_0}$, the leafwise measure $[\mu^P]^{G^\beta}_x$ is atomic for $\mu^P$-a.e. $x$.
\end{lemma}
\begin{proof}
    Take $s\in\ker(\alpha)\cap\overline{\mathcal W^P}$ for which $\mu^P$ is $s$-ergodic, which can be done by Assumption \ref{main assumption}. Since $\mu^P$ is not $G^{-\beta}$-invariant, it follows from Proposition \ref{entropic measure rigidity} that 
\end{proof}

We need the following observation about higher rank, simple Lie algebras, which follows immediately as in this case, $\beta_P^H$ cannot be a simple root.

\begin{lemma}\label{sum of positive roots except biggest}
The sum of all elements of $\Delta_P\setminus\{\beta_P^H\}$  is bigger than or equal to $\beta_P^H$ on $\intW$.
\end{lemma}

\begin{transverse lyapunov exponent strictly maximally negative for least singular elements}
    Let $\alpha\in\Pi_P$ and $s$ be as in the hypothesis of the proposition. We have $\lo(s)<0$ by lemma \ref{kernel of transversal no faces}.

    \begin{transverse lyapunov exponent strictly maximally negative for least singular elements cases}
        $\lo$ is not resonant to a root.
    \end{transverse lyapunov exponent strictly maximally negative for least singular elements cases}    
    We have by theorem \ref{entropy product structure} that
\begin{equation*}
  \hmup(s)=\sum_{\beta\in\Delta_P\setminus\{\alpha\}}\beta(s)\text{, }\hmup(s^{-1})=\underset{-\beta\in\Delta_{Q_0}}{\sum_{\beta\in\Delta_P\setminus\{\alpha\}}}\beta(s)+\hmup(s^{-1}|\mathcal W^{\lo})
\end{equation*}
as 
 \begin{equation*}
    h_{\mu^P}(s^{-1}|\mathcal W^{-\beta}) =
    \begin{cases*}
      \beta(s) & if $-\beta\in\Delta_{Q_0}$ \\
      0        & otherwise
    \end{cases*}
  \end{equation*}
for all $\beta\in\Delta_P\setminus\{\alpha\}$, by  lemma \ref{entropy ergodic components}.
Since $\hmup(s)=\hmup(s^{-1})$ it follows that
\begin{equation*}
  \hmup(s^{-1}|\mathcal W^{\lo})=\underset{-\beta\notin\Delta_{Q_0}}{\sum_{\beta\in\Delta_P\setminus\{\alpha\}}}\beta(s);
\end{equation*}
we then have
\begin{equation*}
  \lo(s^{-1})\geq\hmup(s^{-1}|\mathcal W^{\lo})=\underset{-\beta\notin\Delta_{Q_0}}{\sum_{\beta\in\Delta_P\setminus\{\alpha\}}}\beta(s);
\end{equation*}
but as $|\Delta\setminus\Delta_{Q_0}|\geq 3$,  there exists some positive root with respect to $P$ different to both $\beta_P^H$ and $\alpha$, and since $s\in \operatorname{Int}(\ker(\alpha)\cap\overline{\mathcal W^P})$, we get that $\lo(s^{-1})>\beta_P^H(s)$, which implies the inequality of the proposition. 

\begin{remark}
    The last argument is the only place in the whole proof of proposition \ref{transverse lyapunov exponent strictly maximally negative for least singular elements} where we use that $|\Delta\setminus\Delta_{Q_0}|\geq 3$.
\end{remark}

\begin{transverse lyapunov exponent strictly maximally negative for least singular elements cases}
    $\lo$ is resonant to $-\beta_P^H$.
\end{transverse lyapunov exponent strictly maximally negative for least singular elements cases}

We proved this case of this proposition in the proof of proposition \ref{transverse lyapunov exponent strictly maximally negative} when $r(G)=2$, so we assume for this case that $r(G)\geq 3$. 

In this case we have $Q_0=\operatorname{Stab}(\mu^P)=P$ by remark \ref{resonance to negative of highest root implies stabilizer is P}. We also have that the coarse Lyapunov exponent of $\lo$ equals $\chi=\{\lo,-\beta_P^H\}$.

We have by lemma \ref{entropy ergodic components} that
\begin{equation*}
\hmup(s)=\sum_{\beta\in\Delta_P}\beta(s)\text{, } \hmup(s^{-1})=\hmup(s^{-1}|\mathcal W^\chi);
\end{equation*}
from lemma \ref{no full entropy along coarse Pesin manifold} we obtain
\begin{equation*}
        \lo(s^{-1})+\beta_P^H(s)>\hmup(s^{-1}|\mathcal W^\chi)
        =\sum_{\beta\in\Delta_P}\beta(s);
\end{equation*}
thus by canceling accordingly in the previous inequality, we obtain
\begin{equation*}
\lo(s^{-1})>\sum_{\beta\in\Delta_P\setminus\{\beta_P^H\}}\beta(s),
\end{equation*}
however by lemma \ref{sum of positive roots except biggest} the right hand side of this inequality is greater than or equal to $\beta_P^H(s)$, hence $\lo(s^{-1})>\beta_P^H(s)$, and as $\beta_P^H$ is the highest root with respect to $P$, we obtain the result in this case.

\begin{transverse lyapunov exponent strictly maximally negative for least singular elements cases}
$\lo$ is resonant to a root different than $-\beta_P^H$.
\end{transverse lyapunov exponent strictly maximally negative for least singular elements cases}

Denote by $\beta_0$ the root $\lo$ is resonant to, then $-\beta_0\in\Delta_P$ because of lemma \ref{kernel of transverse exponent touches not weyl chamber}, and $\beta_0\notin\Delta_{Q_0}$ by Remark \ref{resonant root can't be in root complement of Q0}. We have by theorem \ref{entropy product structure} and lemma \ref{entropy ergodic components} that
\begin{equation*}
    \hmup(s)=\sum_{\beta\in\Delta_P}\beta(s)\text{, }\hmup(s^{-1})=\underset{-\beta\in\Delta_{Q_0}}{\sum_{\beta\in\Delta_P}}\beta(s)+\hmup(s^{-1}|\mathcal W^{\chi}),
\end{equation*}
hence, since $h_{\mu^P}(s)=h_{\mu^P}(s^{-1})$ we obtain that
\begin{equation*}
    \hmup(s^{-1}|\mathcal W^{\chi})=\underset{-\beta\notin\Delta_{Q_0}}{\sum_{\beta\in\Delta_P}}\beta(s);
\end{equation*}
by lemma \ref{no full entropy along coarse Pesin manifold} we obtain
\begin{equation*}
        \lo(s^{-1})+\beta_0(s^{-1})>\hmup(s^{-1}|\mathcal W^\chi)
        =\underset{-\beta\notin\Delta_{Q_0}}{\sum_{\beta\in\Delta_P}}\beta(s);
\end{equation*}
thus
\begin{equation*}
    \lo(s^{-1})>\underset{-\beta\notin\Delta_{Q_0}}{\sum_{\beta\in\Delta_P\setminus\{-\beta_0\}}}\beta(s),
\end{equation*}
however, $\beta_H^P\neq-\beta_0$ by hypothesis; $\beta_H^P\neq\alpha$ by lemma \ref{kernel of transverse exponent touches not weyl chamber}; $-\beta_P^H\notin\Delta_{Q_0}$ because of remark \ref{measure can't be invariant under negation of highest root}, and thus the right hand side of the last inequality has $\beta_P^H(s)$ as a summand, which implies $\lo(s^{-1})>\beta_P^H(s)$, and since $\beta_P^H$ is the highest root with respect to $P$, we obtain the proposition in this case.

\end{transverse lyapunov exponent strictly maximally negative for least singular elements}

\begin{definition}
 By propositions \ref{transverse lyapunov exponent strictly maximally negative} and \ref{transverse lyapunov exponent strictly maximally negative for least singular elements} we have that $\lo(a^{-1})>\beta(a)$ for all $a\in\intW$ and all $\beta\in\Delta_P^+$. From we get that for $\mu^P$-a.e. $x$ there exists a one-dimensional global Pesin manifold associated to the Lyapunov exponent $\lo$, which equals the first unstable Pesin manifold  with respect to $a^{-1}$, $\mathcal W^{u,1}(x;a)$ for all $a\in\intW$, see definition \ref{definition ith unstable}; which we denote by $\mathcal W^T(x)$ and call \emph{transverse Pesin manifold at $x$}.
\end{definition}

The name in the previous definition comes from the observation that $\mathcal W^T(x)$ is transversal to the $G$-orbits of the action $G\acts M$ at $x$, because of the description of the Lyapunov exponent $\lo$ of $\mu^P$. Observe that $\mathcal W^T(x)$ might not be contained in $\m$, unless the original given locally free action $G\acts M$ is minimal, even though $\mu^P$ is supported on this minimal set.

\begin{remark}\label{first unstable for regular and not so singular elements}
    If $|\Delta\setminus\Delta_{Q_0}|\geq 3$, then for any $\alpha$, any $s\in\operatorname{Int}(\ker(\alpha)\cap \overline{\mathcal W^P})$ and any $a\in\intW$,  the first unstable manifold $\mathcal W^{u,1}(x;s)$ exists and equals $\mathcal W^{u,1}(x;a)$, i.e. $\mathcal W^T(x)$, for $\mu^P$-a.e. $x$,  because of lemma \ref{uniqueness of ith unstable higher rank} and proposition \ref{transverse lyapunov exponent strictly maximally negative for least singular elements}.
\end{remark}

\section{Construction of immersed submanifold}\label{Proof of theorem when DeltaminusDeltaQgeq3}

Recall $Q_0$ is the parabolic subgroup of $G$ equal to the stabilizer of $\mu^P$. In this section we show that, when $|\Delta\setminus\Delta_{Q_0}|\geq 3$,  there exists an injectively immersed submanifold $N$ of $M$, of the same regularity as $M$, with the following properties:
\begin{enumerate}[label=(\roman*)]
    \item For $\mu^P$-a.e. $x$ the transverse Pesin manifold $\mathcal W^T(x)$ is contained in $N$;
    \item $N$ is $Q_0$ invariant;
    \item $\dim(N)=\dim(Q_0)+1$;
    \item $N$ is an injectively immersed submanifold of $M$;
    \item  $N$ intersects injectively immersed $G$-orbits transversely in $M$;
    \item If $x\in M$ and $G\cdot x\cap N\neq\emptyset$, this intersection equals  $gQ\cdot x$ for some $g\in G$;
    \item for any $g'\in G$ we have that $g'Q\cdot x\cap gQ\cdot x\neq \emptyset$ implies $gQ=g'Q$.
\end{enumerate}
%We shall carry this construction by using theorem \ref{topological measure rigidity}; and then show that the resulting submanifold is closed by using the $P$-unique ergodicity on minimal subsets of our given action $G\acts M$, see proposition \ref{uniquePergodicity}. 

We will make a different construction in the case $r(G)=2$ in Section \ref{Construction of submanifold when r(G)=2}.

For each $\mu^P$ generic enough $x$, we construct an injectively immersed closed submanifold of $M$ containing $x$, which we denote by $N_x$, obtained by taking  the $Q_0$ orbit of the transverse Pesin manifold $\mathcal W^T(x)$; we then prove that, in a set of full $\mu^P$ measure, this construction yields the same manifold, which we denote $N$; and we will show $N$ satisfies the conditions listed above.

In Section \ref{N is a closed embedded submanifold} we prove $N$ is in fact a closed embedded submanifold, which is equivalent to the existence of a smooth $G$-equivariant map $M\rightarrow G/Q$.

\subsection{Essential $P$-equivariance of transverse Pesin manifold}

In this section we prove the following assertion, under the assumption that $|\Delta\setminus\Delta_{Q_0}|\geq 3$. Recall we are working under assumption \ref{main assumption}. 

\begin{proposition}\label{P-equivariance of transverse Pesin manifolds}
    Suppose $|\Delta\setminus\Delta_{Q_0}|\geq 3$. For $\mu^P$-a.e. $x$ and $\haar_P$-a.e. $p$ we have 
    \begin{equation*}
        p\cdot \mathcal W^T(x)=\mathcal W^T(px).
    \end{equation*}
\end{proposition}

\begin{proof}%BE CAREFUL WITH THIS ARGUMENT!!!!!!!!!!!
Fix $a\in\intW$ such that $\mu^P$ is $a$-ergodic. For each $\alpha\in\Pi_P$ fix some $s_\alpha\in\operatorname{Int}(\ker(\alpha)\cap\mathcal W^P)$ such that $\mu^P$ is $s_\alpha$-ergodic, which exists by Assumption \ref{main assumption}.

As $\mu^P$ is $G^\alpha$-invariant, because of proposition \ref{haarconditionals}, we obtain that for $\mu^P$-a.e. $x$ and $\haar_{G^\alpha}$-a.e. $u$, the transverse Pesin manifolds $\mathcal W^T(x;a)$ and $\mathcal W^T(ux;a)$ exists; because of remark \ref{first unstable for regular and not so singular elements} they equal the manifolds $\mathcal W^T(x;s_{\alpha})$ and $\mathcal W^T(ux;s_\alpha)$, respectively; however, $u$ and $s_\alpha$ commute, this implies $$u\mathcal W^T(x;s_{\alpha})=u\mathcal W^{u,1}(x;s_{\alpha})=\mathcal W^{u,1}(ux;s_{\alpha})=\mathcal W^T(ux;s_\alpha),$$ by definition of the first unstable manifold and of transverse Pesin manifold, and hence $$u\mathcal W^T(x;a)=u\mathcal W^T(x;s_{\alpha})=\mathcal W^T(ux;s_\alpha)=\mathcal W^T(ux;a).$$

Since every root in $\Delta_P$ is the sum of elements of $\Pi_P$, we can use the fact that $[G^\alpha, G^\beta]= G^{\alpha+\beta}$, for all roots $\alpha$ and $\beta$ in $\Delta_P$ such that $\alpha+\beta$ is a root, and thus belonging to $\Delta_P$. Since $G$ is split, and $G^\gamma$ is a one dimensional non-compact group, it follows that if $A\subseteq G^\alpha$ and $B\subseteq G^\beta$ have full Haar measure, in the respective groups, \emph{any} $u\in G^{\alpha+\beta}$ can be expressed as $u=[g,h]$ for some $g\in A$ and $h\in B$. Hence, since $\mathcal W^T(ux;a)$ exists for $\mu^P$-a.e. $x$ and $\haar_{G^\alpha}$-a.e. $u$, we get that $u\cdot\mathcal W^T(x;a)=\mathcal W^T(ux;a)$ for $\mu^P$-a.e. $x$ and $\haar_{G^\alpha}$-a.e. $u$, for all $\alpha\in\Delta_P$, as $\mu$ is $G^\alpha$-invariant for such $\alpha$.  

Therefore, since $\mathcal W^T(-;a)$ is $A$-invariant, combining everything we have done with an ordering of the positive roots of $P$, along with Fubini's theorem, we obtain the proposition.
\end{proof}

\subsection{Construction of immersed submanifold $N_x$}
We want to prove first that for any transverse Pesin manifold of a $\mu^P$ generic enough point, the $Q_0$ saturation of this manifold, is an immersed submanifold. This is easily achieved once the following property is shown.

\begin{lemma}\label{transverse Pesin manifold is uniformly locally nice}
   For $\mu^P$-a.e. $x$,  $\mathcal W^T(x)$ intersects each $G$-orbit transversely.
\end{lemma}
\begin{proof}
    
    Pick $a$ such that $\mu^P$ is $a$-ergodic. Suppose Birkhoff ergodic theorem holds for $x$ with respect to $a$ and $\mu^P$.

    There exists $\Lambda\subseteq M$ of positive $\mu^P$ measure such that for some $r>0$ we have that for every $y\in\Lambda$ the following holds: 
    \begin{enumerate}
        \item $\mathcal W^T(y)$ exists;
        \item the closed ball of radius $r$ in $\mathcal W^T(y)$ centered at $y$ is an embedded submanifold that is transversal to the $G$-orbits.
    \end{enumerate}
    The latter property holds in a set of positive $\mu^P$ measure as for $\mu^P$-a.e. $y$, $\mathcal W^T(y)$ is transversal to $G\cdot y$ at $y$,

    Suppose, to get a contradiction, that $\mathcal W^T(x)$ intersects a $G$-orbit non-transversely. Let $z$ be a point of $\mathcal W^T(x)$ for which this submanifold is not trasnversal to $G\cdot z$ at $z$. As the Birkhoff ergodic theorem holds for $x$ with respect to $a$ and $\mu^P$, we get that there exists a large enough $N>0$ such that $a^Nx\in\Lambda$ and $d(a^Nx,a^Nz)<r$. Then the ball of radius $r$ of $\mathcal W^T(a^Nx;a)$ centered at $a^Nx$ intersects $G\cdot a^Nz$ non-transversely, which contradicts (2) above.    
    %Suppose not. Since $\mathcal W^T_a(x)$ is a one-dimensional immersed submanifold, it then follows that there exists some $y$ in it such that $\mathcal W^T(x,a)$ is tangent to a $G$-orbit. Since $\mathcal W^T(x,a)$ is the Pesin manifold corresponding to the Lyapunov exponent $\lambda^T_{\mu^P}(a)$, with respect to the action of $a$ on $M$, it follows from (\ref{Pesin manifolds}) that for either of the unit vectors $v\in T_y\mathcal W^T(x,a)$ we have 
    %\begin{equation*}
     %   \lim_{n\rightarrow\infty}\frac{1}{n}\log||D_ya^n(v)||=\lambda^T_{\mu^P}(a),
    %\end{equation*}
    %but this is a contradiction to the fact that $v\in T_yG\cdot y$, as this containment implies that
    %\begin{equation*}
     %   \lim_{n\rightarrow\infty}\frac{1}{n}\log||D_ya^n(v)||\geq\min\{\beta(a):\beta\in\Delta\}.
    %\end{equation*}
    %The two previous inequalities are contradictory because of proposition \ref{transverse lyapunov exponent strictly maximally negative}.
\end{proof}
\begin{definition}\label{construction submanifold}
For any $x$ as in proposition \ref{P-equivariance of transverse Pesin manifolds}, consider the set
\begin{equation*}
    N_x:=Q_0\cdot\mathcal W^T(x):=\{qy:q\in Q_0, y\in\mathcal W^T(x)\}.
\end{equation*} 
    
\end{definition}

The transversality property yields the following property.

\begin{lemma}\label{Nx is immersed}
    For $\mu^P$-a.e. $x$, $N_x$ can be given the structure of an immersed submanifold of $M$ of dimension $\dim(Q_0)+1$.
\end{lemma}

\begin{proof}
    Let $x$ be $\mu^P$ generic for lemma \ref{transverse Pesin manifold is uniformly locally nice} to hold. Take any $y\in \mathcal W^T(x)$ and $q\in Q_0$.

    Since $q$ is a diffeomorphism $M\rightarrow M$ preserving the $G$-orbits foliation, and $\mathcal W^T(x)$ is transversal to $G\cdot y$ at $y$, we obtain that $D_yq[T_y\mathcal W^T(x)]+T_{qy}G\cdot qy$ is a vector subspace of $T_{qy}M$ of dimension $\dim(Q_0)+1$. Hence, there exists some $J$, precompact neighborhood of $y$ in $\mathcal W^T(x)$ diffeomorphic to $[-1,1]$, along with some $\delta>0$ such that the map $Q_0^{<\delta}\times J\rightarrow M$ given by $(q,z)\mapsto qz$ is an embedding onto its image. 

    All such embeddings endow $N_x$ with an immersed submanifold structure as a subset of $M.$
\end{proof}

\subsection{$N_x$ is injectively immersed}

We use the equivariance property given by proposition \ref{P-equivariance of transverse Pesin manifolds} to show $N_x$ is injectively immersed in $M$.

\begin{lemma}\label{injectivey immersed N}
    For $\mu^P$-a.e. $x$, $N_x$ is an injectively immersed submanifold of $M$, of dimension $\dim(Q_0)+1$.
\end{lemma}

\begin{proof} 
    Let $x$ be $\mu^P$-generic enough for Definition \ref{construction submanifold} to hold. Suppose that for some $q_0\in Q_0$ and some $z_0,z_1\in\mathcal W^T(x)$ we have $q_0z_0=z_1$. The injectively immersed submanifolds $\mathcal W^T(x)$ and $q_0\cdot \mathcal W^T(x)$ intersect at $z_1=q_0z_0$. By proposition \ref{P-equivariance of transverse Pesin manifolds} there exists some $q_1\in Q$ such that both $\mathcal W^T(q_1x)$ and $\mathcal W^T(q_1q_0x)$ exist and
    \begin{equation}\label{wiedhgqwafg84}
        \mathcal W^T(q_1x)=q_1\mathcal W^T(x), \mathcal W^T(q_1q_0x)=q_1q_0\mathcal W^T(x).
    \end{equation}
    Then $q_1z_1=q_1q_0z_0$ belongs to both $\mathcal W^T(q_1x)$ and $\mathcal W^T(q_1q_0x)$, hence these two submanifolds are equal, and in particular by (\ref{wiedhgqwafg84}) this implies $\mathcal W^T(x)=q_0\mathcal W^T(x)$. 
    
    As $N_x=Q_0\cdot \mathcal W^T(x)$, this argument implies the immersed submanifold $N_x$ of $M$ has no self-intersections, i.e., this submanifold is injectively immersed.

    By the construction of $N_x$, and as it is an injectively immersed submanifold, it follows that the sets $U\cdot V$, where $U$ is an open set of $G$ and $V$ is an open set of $\mathcal W^T(x)$ are a basis for the immersed topology of $N_x$, and hence, the dimensional assertion follows.
\end{proof}

\subsection{Intersections of $N_x$ with $G$-orbits}

The following proposition will be the most important tool to prove the properties indicated at the beginning of this section, and will also yield independence of $N_x$ with respect to $x$.

\begin{proposition}\label{coherence proposition}
     For $\mu^P$-a.e. $x_1$ and $x_2$ we have that if there is $z\in\mathcal W^T(x_1)$ such that $gz\in\mathcal W^T(x_2)$ for some $g\in G$, then $g\in Q$.
\end{proposition}

Before proving this proposition, we need some lemmas and definitions. 

\subsubsection{Pertubations by generic elements of $P$ and big cell of $G/P$}

\begin{lemma}\label{Pertubations by generic elements of $P$ have $LU$ factorization}
For any $g\in G$ we have that for $\haar_P$-a.e. $p$, $pg=up'$ for unique $u\in U(P^{\operatorname{op}})$ and $p'\in P$. Furthermore, the $\haar_P$-a.e. defined function $p\mapsto p'$ is absolutely continuous. 
\end{lemma}

\begin{sketch of proof}
    We do the proof only for $G=\SL(n,\R)$, the proof being similar for the other simple split Lie groups. The elements of the Weyl group $W$ of $(G,A)$ can be represented as column permutation matrices. Fix any $w\in W$. It can easily be verified that for a generic $p\in P$, where in this case $P$ is the subgroup of upper triangular matrices of $\SL(n,\R)$, for each $1\leq k\leq n$ we have $\det (pw)_k\neq 0$, where $(pw)_k$ is the upper left-hand corner $k\times k$ submatrix of $pw$, as $pw$ is obtained from $p$ by permuting its columns. Therefore, the lemma is true for all the elements of the Weyl group, by the existence theorem for $LU$ factorization.

    From the above and using the Bruhat decomposition of $G$ with respect to $P$, we obtain the result.
\end{sketch of proof}

\begin{definition}
    Suppose $\alpha$ is a root of $(G,A)$ such that $\alpha\in\Delta_{P}$. We say $g\in G$ \emph{has \emph{nontrivial} $G^{-\alpha}$ component with respect to $P$} if $g$ can be factored as $up$, where $p\in P$ and $u\in P^{\operatorname{op}}$ is such that if $v$ is the element of $\operatorname{Lie}(P^{\operatorname{op}})$ such that $u=\exp(v)$, and $v$ has non trivial $\mathfrak g_{-\alpha}$ component in the decomposition $\operatorname{Lie}(P^{\operatorname{op}})=\bigoplus_{\beta\in\Delta_{P}}\mathfrak g_{-\beta}$.
\end{definition}

\begin{lemma}\label{almost surely nice perturbations}
Let $P\leq Q\leq G$ be a proper parabolic. Suppose $g\in G\setminus Q$. Then there exists a simple positive root $\alpha\in\Pi_P\setminus\Delta_Q$ such that if $pg=up'$, as in lemma \ref{Pertubations by generic elements of $P$ have $LU$ factorization}, then $u$ has nontrivial $G^{-\alpha}$ component for $\haar_P$-a.e. $p$.
\end{lemma}

\begin{proof}
    By lemma \ref{Pertubations by generic elements of $P$ have $LU$ factorization}, choose $p_0\in P$ such that $p_0g=u_1p_1$ for some $u_1\in U(P^{\operatorname{op}})$ and $p_1\in P$. Since $g\notin Q$, it follows that $u_1\in U(P^{\operatorname{op}})\setminus\{e\}$, as $g\notin Q$.

    Then $u_1$ has non-trivial $G^{-\alpha}$ component for some $\alpha\in\Delta_P$, choose such an $\alpha$ minimal with respect to the ordering on the roots of $(G,A)$ with respect to $P$. Choose $\gamma\in \Delta_P\cup\{0\}$, with $\alpha+\gamma\in\Pi_P$. By using commutators, it follows that for all $p\in P$, $pu_1=u_2p_2$, for some $u_2\in U(P^{\operatorname{op}})$; furthermore, $u_2$ will have non-trivial $G^{-\alpha-\gamma}$ component for $\haar_P$-a.e. $p$, in fact, the set set of all $p\in P$ not satisfying this is a proper algebraic subset of $P$ of lower dimension, and hence of $\haar_P$ measure zero. 

    Therefore, as $pp_0g=pu_1p_1=u_2p_2p_1$, we obtain the result.
\end{proof}

\begin{proof of coherence proposition}
    Throughout the proof we indicate how generic, with respect to $\mu^P$, $x_1$ and $x_2$ have to be. For each $\alpha\in\Pi_P\setminus\Delta_{Q_0}$, pick some $s_\alpha\in\operatorname{Int}(\ker(\alpha)\cap\overline{\mathcal W^P})$ such that $\mu^P$ is $s_\alpha$-ergodic, recall we are working under assumption \ref{main assumption}. 
    
    Suppose $x_1$ is such that for all $\alpha\in\Pi_P$ and for $\haar_P$-a.e. $p$ we have the following properties:
    \begin{enumerate}[label=(\roman*)]
        \item proposition \ref{topological measure rigidity} holds for $\mu^P$ and $px_1$ with respect to $s_\alpha$ and $\alpha$, for a.e. $p\in P$, which can be done as $\mu^P$ is $C_G(A)\ltimes\langle N_{s_\alpha},G^{\alpha}\rangle$-invariant, as $P=C_G(A)U(P)$ and $\langle N_{s_\alpha},G^{\alpha}\rangle=U(P)$, because of our choice of $s_\alpha$ with respect to $\ker(\alpha)$ and $\mathcal W^P$, and because of proposition \ref{haarconditionals}.
        \item Birkhoff's ergodic theorem holds for $px_1$ with respect to $s_\alpha$ and $\mu^P$.
    \end{enumerate}
    
    Suppose to get a contradiction that $g\notin Q$. Choose $p_0\in P$  such that the following holds:
    \begin{enumerate}
        \item $p_0g=u_1p_1$ for some $u_1\in U(P^{\operatorname{op}})$ and $p_1\in P$, such that for some $\alpha\in\Pi_P$ with $-\alpha\notin\Delta_{Q_0}$, we have that $u_1$ has nontrivial $G^{-\alpha}$ component, which can be done by lemma \ref{almost surely nice perturbations};
        \item $p_1$ satisfies properties (i) and (ii) above with respect to $x_1$, this can be done by the last part of lemma \ref{almost surely nice perturbations};
        \item  $p_0x_2\in\supp(\mu^P)$;
        \item  $p_0\mathcal W^T(x_2)=\mathcal W^T(p_0x_2)$, which can be done because of remark \ref{first unstable for regular and not so singular elements}, and proposition \ref{P-equivariance of transverse Pesin manifolds}. 
    \end{enumerate}

    For $\alpha$ as in (1), denote $s_\alpha$ by $s$ to ease notation. 
    \s

    As $x_1\in\supp(\mu^P)$, so that $p_1x_1\in\supp(\mu^P)$; and because of (2) and (ii), it follows that there exists an increasing sequence of natural numbers $\{n_k\}_k$ such that $s^{n_k}p_1x_1$ converges to $p_1x_1$ as $k\rightarrow\infty$. 

    We have that $u_1$ has nontrivial $G^{-\alpha}$ component, which implies $s^n u_1s^{-n}$ converges to some $v_1\in G^{-\alpha}\setminus\{e\}$. 
    
    However, by both lemmas \ref{kernel of transverse exponent touches not weyl chamber} and \ref{negativity of transversal lyapunov exponent on positive weyl chambers} we have that $\lambda_{\mu^P}(s)<0$, since $s\in\operatorname{Int}(\ker(\alpha)\cap\overline{\mathcal W^P})$, by remark \ref{first unstable for regular and not so singular elements} we get that since $z\in\mathcal W^T(x_1)=\mathcal W^T(x_1;s)$($z$ as in our hypothesis), 
    \begin{equation*}
        d(s^np_1z,s^np_1x_1)\rightarrow 0  
    \end{equation*}
     which implies $s^{n_k}p_1z\rightarrow  p_1x_1$, as $s^{n_k}p_1x_1\rightarrow p_1x_1$, which in turn implies 
     \begin{equation}\label{kanv8i4k}
     s^{n_k}u_1p_1z=(s^{n_k}u_1s^{-n_k})s^{n_k}p_1z\rightarrow v_1p_1x_1.    
     \end{equation}

     However, $u_1p_1z=p_0gz\in p_0\mathcal W^T(x_2)$, hence, by (4), we obtain $u_1p_1z\in\mathcal W^T(p_0x_2)=\mathcal W^T(p_0x_2,s)$; so that by definition of transverse Pesin manifold it follows that 
     \begin{equation*}
         d(s^nu_1p_1z, s^np_0x_2)\rightarrow 0;
     \end{equation*}
     this, along with (\ref{kanv8i4k}), implies $s^{n_k}p_0x_2\rightarrow v_1p_1x_1$; this in turn proves that $v_1p_1x_1\in\overline{\supp(\mu^P)}=\supp(\mu^P)$, as $p_0x_2\in\supp(\mu^P)$ by (3), and hence $s^np_0x_2\in\supp(\mu^P)$ for all $n$, as $\mu^P$ is $P$-invariant. 

     Therefore, by (2) and (i), and recalling Proposition \ref{topological measure rigidity}, we get that $\mu^P$ is $G^{-\alpha}$-invariant, which is a contradiction, as $-\alpha\notin\Delta_{Q_0}$, because of (1), and as $Q_0=\operatorname{Stab}(\mu^P)$.

     This contradiction proves $g\in Q$, which finishes the proof.
\end{proof of coherence proposition}

\subsection{Construction of submanifold when $r(G)=2$}\label{Construction of submanifold when r(G)=2}

This is the case when $G$ has Lie algebra $\mathfrak{sl}(3,\R)$. We take $P$ to be the closed subgroup of $G$ corresponding to the Lie subalgebra of $\lie(G)=\mathfrak{sl}(3,\R)$ corresponding to the upper triangular matrices of $\SL(3,\R)$. If proposition \ref{transverse lyapunov exponent strictly maximally negative for least singular elements} fails, then by remark \ref{reduction to inequality on kernel}, we must have that $Q_0$ is a maximal parabolic subgroup of $G$ containing $P$. We may assume without loss of generality that $Q_0$ is the Lie subgroup of $G$ generated by $G^{\beta_{2,1}}$ and $P$.

If proposition \ref{transverse lyapunov exponent strictly maximally negative for least singular elements} holds for $\mu^P$ in this case, we are done by remark \ref{reduction to inequality on kernel}. Let us assume this is not the case for the rest of the section. 

By assumption \ref{main assumption} and theorem \ref{pughshub} there exists $s\in\ker(\beta_{2,3})\cap\overline{\mathcal W^P}\setminus\{e\}$ such that $\mu^P$ is $s$-ergodic. 

As $Q_0=\langle G^{\beta_{2,1}}, P\rangle$, and since proposition \ref{transverse lyapunov exponent strictly maximally negative for least singular elements} fails for $\mu^P$, it follows from our choice of $s$ that 

\begin{equation}\label{kasjfg0vb}
    h_{\mu^P}(s)=(\beta_{1,3}+\beta_{1,2})(s)\text{ and }h_{\mu^P}(s^{-1})=\beta_{2,1}(s^{-1})+h_{\mu^P}(s^{-1}|\mathcal W^{\lo}),
\end{equation}
where both equalities follow from $Q_0$-invariance of $\mu^P$; and the last equality follows from theorem \ref{entropy product structure}, along with the fact that, in this case, because of lemma \ref{kernel of transverse exponent touches not weyl chamber}, $\lo$ can only be resonant to $-\beta_P^H=\beta_{3,1}$, for otherwise, $Q_0$ would be $P$ because of remark \ref{resonance to negative of highest root implies stabilizer is P}, which would contradict that $\lo$ is not resonant to a root. As $h_{\mu^P}(s)=h_{\mu^P}(s^{-1})$ and $\beta_{1,2}(s)=\beta_{2,1}(s^{-1})$, we obtain
\begin{equation*}
    h_{\mu^P}(s^{-1}|\mathcal W^{\lo})=\beta_{1,3}(s)
\end{equation*}

Since we are assuming proposition \ref{transverse lyapunov exponent strictly maximally negative for least singular elements} is false on $\ker(\beta_{2,3})$, we get that $\lo(s^{-1})\leq \beta_{1,3}(s)$, however
\begin{equation*}
    \beta_{1,3}(s)=h_{\mu^P}(s^{-1}|\mathcal W^{\lo})\leq\lo(s^{-1})\leq\beta_{1,3}(s),
\end{equation*}
which implies
\begin{equation}\label{maximal entropy for rG=2}
    h_{\mu^P}(s^{-1}|\mathcal W^{\lo})=\lo(s^{-1})
\end{equation}
We have thus proved the following:

\begin{proposition}\label{absolute continuity along transverse Pesin}
    If $\mathfrak g=\mathfrak{sl}(3,\R)$ and the inequality in proposition \ref{transverse lyapunov exponent strictly maximally negative for least singular elements} fails, then $\mu^P$ is equivalent to the Riemannian volume along the Pesin partition $\mathscr W^{\lo}$, with nowhere vanishing smooth Radon-Nikodym derivative.
\end{proposition}
\begin{proof}
    This is a consequence of (\ref{maximal entropy for rG=2}) along with the Ledrappier-Young Theorem, see \cite[Theorem~A]{ledrappieryoung1} and the remark after it.
\end{proof}

Now we prove proposition \ref{coherence proposition} in the case $r(G)=2$. 

\begin{proof of coherence proposition when rG=2}
Let $x_1$ and $x_2$ be $\mu^P$ generic enough. Suppose there are $z\in \mathcal W^T(x_1)$ and $g\in G$ such that $gz\in \mathcal W^T(x_2)$. Suppose, to get a contradiction, that $g\notin Q_0$. 

By proposition \ref{absolute continuity along transverse Pesin} we may take $z_1$ as close to $z$ as necessary in $\mathcal W^T(x_1)$ so that there exists $g_1\in G\setminus Q_0$ with $g_1z_1\in\mathcal W^T(x_2)$, and $z_1$ is generic for $\mu^P$.

Take $p_0\in P$ such that $p_0g_1=u_2p_2$ for some $u_2\in U(P^{\operatorname{op}}$ and $p_2\in P$ such that the following holds:
\begin{enumerate}
    \item $p_2z_1$ is $\mu^P$ generic enough, in particular, such that $p_2z_1\in\supp(\mu^P)$ by Proposition \ref{absolute continuity along transverse Pesin} and Lemma \ref{Pertubations by generic elements of $P$ have $LU$ factorization}.
    \item $u_2$ has nontrivial $G^{-\alpha}$-component for some $\alpha\in \Pi_P$ with $-\alpha\notin\Delta_{Q_0}$, because of lemma \ref{almost surely nice perturbations}.
\end{enumerate}
As $g_1z_1\in\mathcal W^T(x_2)$, by applying proposition \ref{absolute continuity along transverse Pesin}, we obtain $g_1z_1\in \supp(\mu^P)$, and hence $p_0g_1z_1\in\supp(\mu^P)$ by $P$-invariance of $\mu^P$. 

Take $s\in\ker(\alpha)\cap\overline{\mathcal W^P}$ such that $\mu^P$ is $s$-ergodic, which can be done by assumption \ref{main assumption}. 

By (1) above and Birkhoff's ergodic theorem, there exists an increasing sequence $\{n_k\}_k$ such that 
\begin{equation*}
    d(s^{n_k}p_2z_1,p_2z_1)\rightarrow 0,
\end{equation*}
which implies $v_2p_2z_1\in\supp(\mu^P)$, where 
\begin{equation*}
    v_2:=\lim_{n\rightarrow\infty} s^nu_2s^{-n}
\end{equation*}
is such that $v_2\in G^{-\alpha}\setminus\{e\}$, by (2), and since $u_2p_2z_1=p_0g_1z_1\in\supp(\mu^P)$.

By proposition \ref{topological measure rigidity}, we obtain $\mu^P$ is $G^{-\alpha}$-invariant, which contradicts $-\alpha\notin\Delta_{Q_0}$ because of (2). Therefore we obtain the result. 
\end{proof of coherence proposition when rG=2}
\subsection{Independence of $N_x$ on $x$}
Proposition \ref{coherence proposition} yields that $N_x$ does not depend on any $\mu^P$ generic enough $x$.

\begin{proposition}\label{nondependende on x}
    For $\mu^P$-a.e. $x_1$ and $x_2$ we have $N_{x_1}=N_{x_2}$.
\end{proposition}
\begin{proof}
    Fix $a\in\operatorname{Int}(\mathcal W^P)$ such that $\mu^P$ is $a$-ergodic.

    Take any open precompact subset $I_j$ of $x_j$ in the immersed topology of $\mathcal W^T(x_j)=\mathcal W^{u,1}(x_j;a)$ diffeomorphic to $[-1,1]$, for $j=1,2$. 

    We may assume $x_j$ is $\mu^P$ generic enough so that $x_j\in\supp(\mu^P)$ and Birkhoff ergodic theorem holds for $x_j$ with respect to $a$ and $\mu^P$, for $j=1,2$, so that $\{a^nx_j\}_n$ returns arbitrarily close to $x_j$ infinitely many times. Since $\mathcal W^{u,1}(x_j;a)$ is contracted by $a$, it follows that for \emph{every} $y\in\mathcal W^{u,1}(x_j,a)$ there exists $g\in G$ and $z\in I_j$ such that $gz=y$, and hence by definition \ref{construction submanifold}, we get
    \begin{equation}
        N_{x_1}\subseteq G\cdot\mathcal W^T(x_2)\text{ and }N_{x_2}\subseteq G\cdot\mathcal W^T(x_1),
    \end{equation}
    which along with proposition \ref{coherence proposition} yields the proposition.
    
\end{proof}

We now can set the following notation.

\begin{definition}
    Set $N$ to be the injectively immersed submanifold $N_x=Q_0\cdot \mathcal W^T(x)$, for any $x$ generic enough with respect to $\mu^P$ to satisfy proposition \ref{nondependende on x}.
\end{definition}

%We have the following fact.

%\begin{proposition}
 %   $\supp(\mu^P)\subseteq N$.
%\end{proposition}

%\begin{proof}
 %   Since $N$ is obtained by taking [FINISH THIS PROOF, EASY from previous proposition]
%\end{proof}

\section{Construction of smooth equivariant map $M\rightarrow G/Q$}\label{N is a closed embedded submanifold}

In this section we construct the smooth equivariant map $\pi:M\rightarrow G/Q$ of our main result, Theorem \ref{main}, under assumption \ref{main assumption}. The main argument to doing this, appearing in this section, is showing that $N$ intersects all orbits of $G\acts M$.

\subsection{$N$ intersects all $G$-orbits}\label{N intersects all G-orbits section}

The final tool required to construct the equivariant map $M\rightarrow G/Q$ is the following fact.

\begin{lemma}\label{N intersects all orbits}
    $N$ intersects all $G$-orbits
\end{lemma}

To prove this lemma we need some technical observations.

Let us denote by $m^T_x$ the Riemannian volume on the injectively immersed submanifold $\mathcal W^T(x)$, for $\mu^P$-a.e. $x$. Pick any $a\in\intW$ such that $\mu^P$ is $a$-ergodic.

\begin{lemma}\label{Pattraction}
    For $\mu^P$-a.e. $x$ we have that 
    \begin{equation*}
        \lim_{N\rightarrow\infty}\frac{1}{N}\sum_{j=0}^{N-1}\delta_{a^jgy}=\mu^P
    \end{equation*}
    in the weak-$\star$ topology, for $m^T_x$-a.e. $y$ and a.e. $g\in G$. 
\end{lemma}

\begin{proof}
    There are two cases:

    \begin{case empirical measures sequence}
    $|\Delta\setminus\Delta_{Q_0}|\geq 3$.
    \end{case empirical measures sequence}

    In this case, we prove that in fact the lemma holds for \emph{all} $y\in \mathcal W^T(x)$ for $\mu^P$-a.e. $x$. Indeed, if $x$ is $\mu^P$ generic enough for proposition \ref{P-equivariance of transverse Pesin manifolds} and Birkhoff's ergodic theorem to hold for $px$, with respect to $\mu^P$ and $a$, for $\haar_P$-a.e. $p$, then for any $y\in\mathcal W^{u,1}(x;a)$ we have $py\in p\cdot \mathcal W^T(x)=\mathcal W^T(px)\subseteq\mathcal W^s(px;a)$ for $\haar_P$-a.e. $p$, and thus 
    \begin{equation*}
        \lim_{N\rightarrow\infty}\frac{1}{N}\sum_{j=0}^{N-1}f(a^jpy)=\lim_{N\rightarrow\infty}\frac{1}{N}\sum_{j=0}^{N-1}f(a^jpx)=\int fd\mu^P,
    \end{equation*}
    for all $f\in C^0(M)$, as such functions are uniformly continuous; which implies that for all $u\in U(P^{\operatorname{op}})$ we have
    \begin{equation*}
        \lim_{N\rightarrow\infty}\frac{1}{N}\sum_{j=0}^{N-1}f(a^jupy)=\lim_{N\rightarrow\infty}\frac{1}{N}\sum_{j=0}^{N-1}f(a^jpy)=\int fd\mu^P
    \end{equation*}
    for $\haar_P$-a.e. $p$, as $U(P^{\operatorname{op}})\leq N_a$, since $a\in\intW$.

    Since $\haar_G$-a.e. $g$ can be written uniquely in the form $up$, for some $u\in U(P^{\operatorname{op}})$ and $p\in P$ by taking the biggest cell in Bruhat's decomposition, see for instance \cite[Lemma~5.1.4]{zimmerbook}, we get the lemma in this case by applying Fubini's theorem, as the map $g\mapsto p$, where $g=up$, is analytic where defined, and therefore absolutely continuous.

    \begin{case empirical measures sequence}
    $|\Delta\setminus\Delta_{Q_0}|=2$.
    \end{case empirical measures sequence}

    In this case, if Proposition \ref{P-equivariance of transverse Pesin manifolds} holds in this case, then the lemma follows from the argument for the previous case, using Remark \ref{reduction to inequality on kernel}.

    Suppose proposition \ref{P-equivariance of transverse Pesin manifolds} does not hold. Let us denote by $(\mu^P)^T_x$ the corresponding decomposition of $\mu^P$ along the partition $\mathscr W^T$; by proposition \ref{absolute continuity along transverse Pesin}, $(\mu^P)^T_x$ is equivalent to Riemannian volume on $\mathcal W^T(x)$ for $\mu^P$-a.e. $x$ with nowhere vanishing smooth Radon-Nikodym derivative. We may assume $x$ is $\mu^P$ generic enough, so that Birkhoff's ergodic theorem holds for $py$ with respect to $\mu^P$ and $a$, for $(\mu^P)^T_x$-a.e. $y$, and $\haar_P$-a.e. $p$, as $\mu^P$ is $P$-invariant. By making an analogous argument to the one for the previous case, we obtain that for all $f\in C^0(M)$,
    \begin{equation*}
        \lim_{N\rightarrow\infty}\frac{1}{N}\sum_{j=0}^{N-1}f(a^jgy)=\int fd\mu^P,
    \end{equation*}
     for $\haar_G$-a.e. $g$, $(\mu^P)^T_x$-a.e. $y$, and $\mu^P$-a.e. $x$, which proves the lemma for this case.
\end{proof}

Let $\mathcal M_1,\ldots,\mathcal M_r$ be the minimal subsets of $G\acts M$ appearing in proposition \ref{uniquePergodicity}, along with corresponding unique $P$-invariant probability measures $\mu_1,\ldots,\mu_r$ on $\mathcal M_1,\ldots,\mathcal M_r$ respectively, and let us take $\mu^P=\mu^1$, because of assumption \ref{main assumption}.

\begin{lemma}
For any $x\in M$ there exists $j=1,\ldots,r$ such that for any $y\in\supp(\mu_j)$ and any $\epsilon>0$ there exists $p\in P$ such that $d_M(px,y)<\epsilon$.\end{lemma}

\begin{proof}
    Denote by $\lambda_P$ a left invariant Haar measure on $P$. Let $\{F_n\}_n$ be a F\o lner sequence of $P$. By proposition \ref{uniquePergodicity}, it follows that for any $x\in M$ there exists an increasing sequence $\{n_k\}_k$ such that 
    \begin{equation*}
        \frac{1}{\lambda_P(F_{n_k})}\int_{F_{n_k}}\delta_{px}d\lambda_P(p)\rightarrow \mu
    \end{equation*}
    in the weak-$\star$ topology, for some $P$-invariant probability measure $\mu$ on $M$. There are $\rho_1,\ldots,\rho_r$ nonnegative real numbers such that $\rho_1+\ldots+\rho_r=1$ and $\mu=\rho_1\mu_1+\cdots+\rho_r\mu_r$. If $j\in\{1,\ldots, r\}$ is such that $\rho_j\neq 0$, then for any $y\in\supp(\mu_j)$ and any $U$ neighborhood of $y$ in $M$, we get that there exists $p\in P$ such that $px\in U$, because of the weak-$\star$ limit above, and as $\mu_j(U)>0$. Therefore, our assertion follows.
\end{proof}

\begin{proof of N intersects all orbits}
    Pick any $a\in\intW$ such that $\mu_j$ is $a$-ergodic for all $j=1,\ldots, r$. Choose any $x$, that is $\mu^P$ generic enough, so that $N=Q_0\cdot \mathcal W^T(x;a)$, because of Proposition \ref{nondependende on x}, and $x\in\supp(\mu^P)$. Set
    \begin{equation*}
        X=\{x\in M: gx\in N, \text{ for some }g\in G\}
    \end{equation*}
    $X$ is open by construction of $N$ and lemma \ref{transverse Pesin manifold is uniformly locally nice}. We want to show $X$ is closed as well, as $M$ is connected, this will yield the result. 

    Take any $z\in\overline{X}$. Fix $j\in\{1,\ldots, r\}$ such that lemma \ref{Pattraction} holds for $z$; so in particular, by our choice of $x$, we have that for any $\delta>0$ there is $p\in P$ such that $d(pz,x)<\delta$.

    Fix any $y\in\supp(\mu_j)$ generic enough for $\mu_j$.

    Take any $f\in C^0(M)$ and any $\epsilon>0$. Choose $\delta=\delta(y,f,1,\epsilon)>0$ for proposition \ref{local rigidty of Pinvariant measures} to hold for  the measure $\mu_j$.

    There exists $p\in P$ such that $d(pz,y)<\delta$. As $z\in\overline{X}$, there is $w\in\mathcal W^T(x)$ such that $d(G\cdot w,z)<\delta$, and $w$ such that 
    \begin{equation*}
        \lim_{N\rightarrow\infty}\sum_{k=0}^{N-1}f(a^kgw)=\int fd\mu^P
    \end{equation*}
    for $\haar_G$-a.e. $g\in G$ by Lemma \ref{Pattraction}.

    By our choice of $\delta$ above we have that since $d(G\cdot w,y)<\delta$, it follows, using the argument in Case 1 in the proof of Lemma \ref{Pattraction}, that there exists an subset $V$ of $G\cdot w$, of positive $\haar_G$ measure, such that for all $v\in V$ we have that 
    \begin{equation*}
        \left |\limsup_{N\rightarrow\infty}\frac1N\sum_{k=0}^{N-1}f(a^kvw)-\int fd\mu_j\right |<\epsilon.
    \end{equation*}
    Combining this inequality with the previous equality we obtain
    \begin{equation*}
        \left |\int fd\mu^P-\int fd\mu_j\right|<\epsilon
    \end{equation*}
    Since $f\in C^0(M)$ and $\epsilon>0$ were arbitrary, it follows that $\mu_j=\mu^P$. Hence, by our choice of $\mu_j$,  there exists $p\in P$ such that $pz$ is close enough to $x$ so that $G\cdot z$ intersects $N$, i.e., $z\in X$.

    This shows $X$ is a closed subset, so $X=M$ by connectedness of $M$, which finishes the proof.    
\end{proof of N intersects all orbits}

%Lemmas \ref{N intersects all orbits} and \ref{Pattraction} imply that $P\acts M$ is uniquely ergodic. 

%\begin{lemma}\label{P unique ergodicity}
 %   $\mu^P$ is the only $P$-invariant measure of $M$.
%\end{lemma}
%\begin{proof}
%Let $\mu$ be any $P$-ergodic probability measure on $M$. By the Mautner's phenomenon, see \cite[Lemma~3.6]{bm}, we get that $\mu$ is $A$-ergodic, and therefore, by Theorem \ref{pughshub}, there exists $a\in\intW$ such that both $\mu$ and $\mu^P$ are $a$-ergodic. 

%Choose $y$ such that Proposition \ref{local rigidty of Pinvariant measures} holds for $y$ with respect to $a$ and $\mu$.

%Let $x$ be generic enough for $\mu^P$. By Lemma \ref{N intersects all orbits} and Proposition \ref{nondependende on x} we get that the set $D$ of all points $z\in M$ such that 
 %   \begin{equation*}
  %      \lim_{N\rightarrow\infty}\frac{1}{N}\sum_{j=0}^{N-1}\delta_{a^jz}=\mu^P
   % \end{equation*}
%in the weak-$\star$ topology, is dense in $M$. 

%\end{proof}

\subsection{Construction of smooth equivariant map}\label{Construction of smooth equivariant map}

We now finish the proof of the main result of this paper, Theorem \ref{main}, that is, we show item (\ref{suspension statement}) of this theorem. This will follow immediately once we show $M$ is $G$-equivariantly diffeomorphic to $G\times_{Q_0} N$. Let us first define the latter object.  

\begin{definition}\label{suspension space defintion}
    Suppose $L$ is a Lie group, with $H\leq L$ a closed subgroup. Suppose $\alpha:H\acts M$ is a left $C^r(r\geq 0)$ action on a $C^r$ manifold. We have two actions
    \begin{equation*}
        L\acts L\times M\actsright H
    \end{equation*}
    given by $l\cdot (l',x)=(ll',x)$ and $h\cdot (l',x)=(l'h^{-1},hx)$, for all $l\in L, h\in H$ and all $(l',x)\in L\times M$. These actions commute, and hence, when we take the quotient of the action by $H$, $M^\alpha:=L\times M/H$, we obtain a $C^r$ manifold, along with an $C^r$ action $L\acts M^\alpha$. Furthermore, the canonical projection map $M^\alpha\rightarrow L/H$ is an equivariant smooth bundle, with fibers diffeomorphic to $M$.
\end{definition}

\begin{proof of main}
    It is enough to show $M$ is equivariantly diffeomorphic to $G\times_{Q_0} N$.

    The map $j:G\times N\rightarrow M$ given by $(g,x)\mapsto gx$ is smooth, and onto $M$ because of Lemma \ref{N intersects all orbits}.

    $j$ is a submersion as  $D_{(g,x)}j$ sends $T_gG\times \{0\}\subseteq  T_{(g,x)}G\times N$ to $T_{gx}(G\cdot x)$, and sends $\{0\}\times T_xN$ to $D_xg[T_xN]$, which is transversal to $T_{gx}(G\cdot x)$, as $g$ is a diffeomorphism preserving the $G$-orbits.

    $j$ descends to a map $\tilde j:G\times_{Q_0} N\rightarrow M$, as $j(gq^{-1},qx)=j(g,x)$, still onto $M$ as $j$ and $\tilde j$ have the same image. The same argument as above shows $\tilde j$ is a submersion. 
    
    However, $\dim(G\times_{Q_0} N)=\dim(G/Q_0)+\dim(N)=\dim(G)-\dim(Q_0)+\dim(Q_0)+1=\dim(G)+1=\dim(M)$, and the tangent space of the orbit of $x$ under the action $Q_0\acts G\times N$ defining the suspension space $G\times_{Q_0} N$ is contained in $\ker D_{(g,x)}j$, this shows $\tilde j$ is a local diffeomorphism. 

    If $(g_1,x_1),(g_2,x_2)\in G\times N$ are such that $g_1x_1=j(g_1,x_1)=j(g_2x_2)=g_2x_2$, then $g_2^{-1}g_1x_1=x_2$, which implies $g_2^{-1}g_1\in Q_0$ by proposition \ref{coherence proposition}, which implies $[(g_1,x_1)]=[(g_2,x_2)]$, and thus, $\tilde j$ is injective, and therefore, $\tilde j$ is a diffeomorphism.
\end{proof of main}

\printbibliography
\end{document}